\documentclass[a4paper,fleqn]{cas-sc}

\usepackage[numbers,sort&compress]{natbib}
\usepackage[ruled]{algorithm2e}
\usepackage{amsthm}
\usepackage{pgfplots}
\usepackage{subcaption}
\usepackage{relsize}
\usepackage{placeins}

\DeclareUnicodeCharacter{2217}{*}
\pgfplotsset{compat=1.18}
\newtheorem{theorem}{Theorem}
\newdefinition{remark}{Remark}

\renewcommand{\printorcid}{}

\begin{document}
\let\WriteBookmarks\relax
\renewcommand{\textfraction}{.05}
\renewcommand{\floatpagefraction}{.7}
\setcounter{topnumber}{5}
\setcounter{bottomnumber}{5}
\setcounter{totalnumber}{10}

\shorttitle{An Iterative Active Subspace Approach for Parametric Model Order Reduction}
\shortauthors{Wang et al.}
\title[mode=title]{An Iterative Active Subspace Approach for Model Order Reduction of Parametric Systems with High-Dimensional Parameter Spaces}

\author[1,2]{Chenzi Wang}
\credit{Conceptualization, Methodology, Software, Formal analysis, Investigation, Validation, Visualization, Writing -- original draft, Writing -- review \& editing}

\author[1]{Peizhi Yu}
\credit{Software, Investigation, Validation}

\author[2]{Lihong Feng}
\credit{Conceptualization, Formal analysis, Supervision, Writing -- review \& editing}

\author[2]{Peter Benner}
\credit{Resources, Writing -- review \& editing}

\author[1]{Wenshuai Lu}
\credit{Resources}

\author[1]{Zheng You}
\credit{Conceptualization, Supervision}

\affiliation[1]{organization={Department of Precision Instrument, Tsinghua University},
            city={Beijing},
            postcode={100084},
            country={China}}

\affiliation[2]{organization={Department of Computational Methods in Systems and Control Theory, Max Planck Institute for Dynamics of Complex Technical Systems},
            city={Magdeburg},
            postcode={39106},
            country={Germany}}

\begin{abstract}
The increasing complexity in design and manufacturing has driven the need for advanced techniques for fast modeling problems with large-dimensional parameter spaces. Avoiding high-fidelity finite element models while achieving fast and accurate simulations in such contexts is challenging. Parametric model order reduction (pMOR) has drawn significant attention in recent years. Nevertheless, the curse of dimensionality in parameter spaces has severely limited its effectiveness. The active subspace (AS) approach has been successfully applied to pMOR for systems with many parameters. However, the balance between accuracy and compactness of the reduced model remains problematic for such systems with high-dimensional parameter spaces. It often results in models that are either not small enough or not accurate enough. In this paper, we propose an iterative active subspace (IAS) approach for parametric model order reduction, which, to some extent, addresses the trade-off between accuracy and reduced model size and achieves substantial computational gains compared to the original active subspace method.
\end{abstract}


\maketitle
\section{Introduction}
\label{sec1}
The increasing complexity and diversification in design and manufacturing have driven the demand for mathematical modeling of problems with large-dimensional parameter spaces. In practical applications, analyzing just a single parameter or a few parameters is no longer adequate to meet specific design requirements. For instance, many structural design issues require optimization involving numerous parameters. Traditional computations based on the finite element method (FEM) are often restricted to parametric sweeps involving only one or a few parameters~\cite{GenMSetal13}. Consequently, design based on parametric modeling always depends heavily on the engineer's expertise.

The challenges of simulation and design optimization for problems with high-dimensional parameter spaces encompass two main aspects. Firstly, the simulation of physical fields largely depends on the finite element method, resulting in high-fidelity models with very high dimensionality, often reaching $10^4$ to $10^6$ degrees of freedom (DoF) in the solution space. This leads to significant computational complexity. Additionally, the parameter spaces are high-dimensional due to the presence of many undetermined parameters (often in the tens or hundreds). This results in the ``curse of dimensionality"~\cite{ForSK08,AshCJetal15}, causing an exponential increase in the number of high-fidelity model evaluations as the number of parameters grows. As a result, conventional methods become nearly infeasible for such problems with large-dimensional parameter spaces.

PMOR~\cite{morBenGW15} is a computational framework that aims to alleviate the burden of simulating high-fidelity systems by constructing compact surrogate models~\cite{morBenOCetal17}. By projecting the original parameter-dependent system onto a reduced subspace spanned by carefully selected basis functions, pMOR preserves input-output behavior while significantly accelerating simulations~\cite{morBenGW15}. Typical projection-based pMOR approaches~\cite{morBenOCetal17,morBenGQetal21a,morBenGQetal21b} include multi-moment-matching, proper orthogonal decomposition, and reduced basis methods, which have demonstrated success in applications ranging from fluid dynamics to structural mechanics~\cite{morFenYBetal16,morLasMQetal14}. However, traditional pMOR techniques still face limitations when handling high-dimensional parameter spaces, as the required number of training samples grows rapidly with parameter dimensionality~\cite{morHesSZ14,morCheG19}.

Recent advances in pMOR have explored coupling with parameter dimension reduction techniques, such as the active subspaces method~\cite{morCon15}. The active subspace method identifies dominant linear subspaces in the parameter domain through gradient-based covariance analysis, effectively decoupling influential parameters from less significant ones~\cite{morConDW14}. This hybrid approach has shown promise in mitigating the curse of dimensionality by first reducing the dimension of the parameter space before constructing reduced-order models (ROMs)~\cite{morTezBR18,morTezDGetal18}. However, this approach faces a key challenge: reducing parameter dimensions shrinks the ROM size but increases approximation errors. For real engineering problems with many parameters, the existing active subspace method often fails to balance accuracy and model compactness effectively. The resulting ROMs tend to be either too large for practical use or too inaccurate for reliable predictions, which is a critical limitation in large-scale applications.

In this paper, we propose an iterative active subspace (IAS) method that builds reduced-order models by repeatedly applying the active subspace method to the error system of the current ROM approximation. At each iteration, one small sub-ROM is generated. With the error rapidly decreasing through iterations, the final IAS-ROM emerges from accumulating the sub-ROMs at previous iterations. This approach effectively balances model size and accuracy while demonstrating superior computational efficiency compared to the original active subspace method.

Section 2 provides a brief introduction to projection-based MOR, the original active subspace method, and its use in pMOR. Section 3 details our proposed iterative active subspace method (IAS) and its application to pMOR. In Section 4, we validate our method using two mechanical models with numerous parameters. Conclusions are drawn in Section 5.

\section{Preliminaries}
In this section, the high-dimensional parameter problem and the projection-based pMOR method are introduced. Following this, the standard active subspace method and its application to pMOR are presented.
\subsection{Problem setting and pMOR}

We consider parametric linear time-invariant (LTI) systems expressed in the following three fundamental representations. The second-order formulation:
\begin{align}
    \begin{array}{rl}
        M(\mu)\ddot x(\mu,t)+D(\mu)\dot x(\mu,t)+K(\mu) x(\mu,t)&=B(\mu)u(t),\\
        y(\mu,t)&=C(\mu)x(\mu,t),
    \end{array}
    \label{eq:second_order_pLTI}
\end{align}
where $x(\mu,t) \in \mathbb{R}^{n} ~\text{ is the state vector}, \mu=[\mu_1,\ldots,\mu_{n_\mu}]^T \in \mathcal P \subset \mathbb R^{n_\mu}$ is the vector of parameters. \(M(\mu),D(\mu),K(\mu) \in \mathbb R^{n \times n}\) are the mass, damping, and stiffness matrices, respectively. \(B(\mu) \in \mathbb{R}^{n\times n_I}, C(\mu) \in \mathbb{R}^{n_O\times n}\) are the input and output matrices, respectively.

For systems already given in first-order form, we use:
\begin{align}
    \begin{array}{rl}
        G(\mu )\dot x(\mu,t) &= A(\mu )x(\mu,t) + B(\mu )u(t),\\
        y(\mu,t) &= C(\mu )x(\mu,t),
    \end{array}
    \label{equ:pLTI}
\end{align}
where  $ A(\mu), G(\mu) \in \mathbb{R}^{n\times n}, B(\mu) \in \mathbb{R}^{n\times n_I}$. 

Sometimes, we are also interested in the static system:
\begin{align}
\begin{array}{rl}
     K(\mu) x(\mu)&=B(\mu),  \\
     y(\mu)&=C(\mu)x(\mu),   
\end{array}
\label{equ:std_sys}
\end{align}
where $x(\mu) \in \mathbb{R}^{n} ~\text{ is the steady-state solution}, K(\mu) \in \mathbb{R}^{n\times n}$,  $B(\mu) \in \mathbb{R}^{n\times n_I}$ $(n_I=1)$.

These formulations typically feature high-dimensional state spaces (\( n \sim 10^4\text{--}10^6 \)) and large parameter dimensionality (\( n_\mu \sim 10\text{--}10^3 \)), presenting significant computational challenges for analysis and control synthesis.

The systems in~(\ref{eq:second_order_pLTI}),~(\ref{equ:pLTI}) or~(\ref{equ:std_sys}) are often referred to as full-order model (FOM). With the typical projection-based pMOR, we can sample in the parameter space and get a set of parameter points $\{\mu\}=\{\mu^1,\mu^2,\ldots,\mu^{M_V}\}$, via, e.g., a greedy process. At every sample point $\mu^k \in \mathbb{R}^{n_\mu}$, a local reduced basis $V_k$ can be constructed~\cite{morBenGW15,morBenOCetal17,morFenYBetal16,morBenGQetal21c}. By orthogonalizing the columns of all the local matrices, we obtain the global projection matrix $V=\mathrm{orth}([V_1,V_2,\ldots,V_{M_V}]) \in \mathbb{R}^{n \times r}$, where orth() means orthogonalization via, e.g., QR decomposition or modified Gram-Schmidt process (MGS). The reduced order model (ROM) for~(\ref{eq:second_order_pLTI}) can be represented as:
\begin{align}
    \begin{array}{rl}
        M_r(\mu)\ddot x_r(\mu,t)+D_r(\mu)\dot x_r(\mu,t)+K_r(\mu) x_r(\mu,t)&=B_r(\mu)u(t),\\
        y_r(\mu,t)&=C_r(\mu)x_r(\mu,t),
    \end{array}
    \label{equ:pLTI_sec_order_ROM}
\end{align} where $M_r(\mu)=V^TM(\mu )V$, $D_r(\mu)=V^TD(\mu )V$, $K_r(\mu)=V^TK(\mu )V$, $B_r(\mu)=V^TB(\mu )$, $C_r(\mu)=C(\mu )V$.

\vspace{1ex}Similarly, the ROM for~(\ref{equ:pLTI}) can be derived as:
\begin{align}
    \begin{array}{rl}
        G_r(\mu )\dot x_r(\mu,t) &= A_r(\mu )x_r(\mu,t) + B_r(\mu )u(t),\\
        y_r(\mu,t) &= C_r(\mu )x_r(\mu,t),
    \end{array}
    \label{equ:pLTI_ROM}
\end{align}
where $G_r(\mu)=V^TG(\mu )V,A_r(\mu)=V^TA(\mu )V$.

\vspace{1ex}Moreover, the ROM for the static system~(\ref{equ:std_sys}) can be written as:
\begin{align}
    \begin{array}{rl}
         K_r(\mu) x_r(\mu)&=B_r(\mu), \\
         y(\mu)&=C_r(\mu)x_r(\mu),
    \end{array}
    \label{equ:pLIT_std_ROM}
\end{align}
where $K_r(\mu)=V^TK(\mu )V,B_r(\mu)=V^TB(\mu)$.

However, with the increment of the number of parameters $n_\mu$, more samples are needed in the parameter space, in order to get an accurate approximation for the original model. While the number of parameters $n_\mu$ reaches $10^{1}-10^{3}$ in many engineering problems, the amount of samples $M_V$ can be really large, resulting in a very high-rank global projection matrix $V$. As a result, the ROM size $r$ is not small enough; such that the computational gain for pMOR can be limited.

Therefore, how to get small and accurate ROMs for models with large-dimensional parameter spaces is crucial for surrogate modeling methods to be deployed in engineering practices.

\subsection{Active subspace method}

In order to deal with high-dimensional parameter spaces, Active Subspace Methods (AS) have gained popularity for performing subspace sensitivity analysis on black-box functions. This method involves identifying the most influential directions (subspaces) along which a scalar function of many variates exhibits significant variations~\cite{morCon15,RomTLetal22,RomTR24,morTezRR22}.
Considering a nonlinear function $f(\mu)$, the process of identifying the active subspace is shown in Algorithm \ref{alg:Active subspace computation}.

\begin{algorithm}
    \caption{Active subspace computation}
    \label{alg:Active subspace computation}
    \textbf{Input:} A multivariate function $f(\mu) \in \mathbb R, \mu:=(\mu_1,\ldots,\mu_{n_\mu})^T\in \mathbb{R}^{n_\mu}$, a density function $\rho$ defining the distribution of $\mu$, and active subspace dimension $r_\mu$.
    
    \textbf{Output:} active subspace $\mathcal U$.
    
    1. Draw \( M \) samples \( \{\mu^j\}, j=1,\ldots,M, \mu^j:=({\mu_1}^j,\ldots,{\mu_{n_\mu}}^j)^T \in \mathbb{R}^{n_\mu} \)
    
    \hspace{3ex}independently according to the density function \( \rho \).

    2. For each \( \mu^j \), compute \( \nabla_{\mu^j} f := \nabla f(\mu^j) \in \mathbb{R}^{n_\mu},j=1,\ldots,M\).
    
    \hspace{3ex}Here, $\forall \mu$, $\nabla f(\mu):=(\frac{\partial f}{\partial \mu_1},\ldots,\frac{\partial f}{\partial \mu_{n_\mu}})^T$.
    
    3. Compute: \vspace{-4ex} \[\hat{C} = \frac{1}{M} \sum_{j=1}^{M} (\nabla_{\mu^j} f)\cdot(\nabla_{\mu^j} f)^T, \space \hat{C} \in \mathbb R^{n_\mu \times n_\mu} \].
    \vspace{-3ex}
    
    4. Compute the eigendecomposition \( \hat{C} = W\Lambda W^T \), $W \in \mathbb R^{n_\mu \times n_\mu}$

    5. $\mathcal U=\operatorname{colspan}(U)$, $U=W_1, W=[W_1,W_2], W_1 \in \mathbb R^{n_\mu \times r_\mu}$,
    
    \hspace{3ex}$\operatorname{colspan}(\cdot)$ means the subspace spanned by the columns of $\cdot$.
    
\end{algorithm}

Steps 3-4 in Algorithm \ref{alg:Active subspace computation} are equivalent to computing the singular value decomposition (SVD) of the matrix~\cite{morCon15}

\begin{align}
     \frac{1}{\sqrt{M}} \begin{bmatrix}
    \nabla_{\mu^1}f & \ldots & \nabla_{\mu^M}f
    \end{bmatrix}
    = W\xi V^T, \text{where} ~\xi^2 = \Lambda.
    \label{eq:svd as}
\end{align}

The dominant left singular vectors $U=W_1$ span the active subspace $\mathcal{U}$ and $\operatorname{colspan} \{W_2\}$ is defined as the inactive subspace, where $\operatorname{colspan}()$ means the subspace spanned by the columns of $W_2$. The function $f(\mu)$ can be approximated as:
\begin{align}
    \begin{array}{llll}
        f(\mu)&=f(WW^T\mu)\\
        &=f(W_1W_1^T\mu+W_2W_2^T\mu)\\
        &\approx f(W_1 \mu_r),
    \end{array}
    \label{eq:fx}
\end{align} 
where $\mu_r=W_1^T\mu\in \mathbb R^{r_\mu}, r_\mu \ll n_\mu$. Sampling of $\mu$ in the whole $n_\mu-$dimensional parameter space $\mathcal P$ can then be implemented by only sampling $\mu_r$ in the $r_\mu-$dimensional active subspace, which greatly reduces the computational complexity and breaks the curse of dimensionality to some extent.

\subsection{Active subspace for parametric Model Order Reduction}

The active subspace approach has been successfully applied to pMOR for parametric systems in~\cite{morConDW14,morTezBR18,SreORetal23}. Algorithm 2 presents the active subspace method applied to the transfer function of (\ref{eq:second_order_pLTI}), (\ref{equ:pLTI}) or (\ref{equ:std_sys}):
\begin{align}
    H(\mu,s) = C(\mu)Q(\mu,s)^{-1}B(\mu), 
    \label{eq:trans_func}
\end{align}

\noindent where $s$ is the Laplace variable, $Q(\mu,s)=s^2M(\mu)+sD(\mu)+K(\mu)$ for~(\ref{eq:second_order_pLTI}), $Q(\mu,s)=sG(\mu)-A(\mu)$ for~(\ref{equ:pLTI}) and $Q(\mu,s)=K(\mu)$ for~(\ref{equ:std_sys}).

Using the active subspace method, we create a sample set in the parameter space \( \{(\mu^j,s^j)\}, j=1,\ldots,M, \mu^j \in \mathbb{R}^{n_\mu} \) and compute \(\nabla_{\mu}H_{oi}(\mu^j,s^j) \in \mathbb R^{n_\mu}\) which is the gradient of $H_{oi}(\mu,s^j)$ w.r.t. $\mu$ and evaluated on $\mu^j$. Here, $H_{oi}(\mu,s)$ is the $o, i$-th entry of $H(\mu,s), o=1,\ldots,n_O,i=1,\ldots,n_I$. Each gradient $\nabla_\mu H_{oi}(\mu^j, s^j)$ constitutes a covariance matrix: 
$\left(\nabla_\mu H_{oi}(\mu^j, s^j) \right) \cdot \left( \nabla_\mu H_{oi}(\mu^j, s^j) \right)^T$. The matrix $\hat C^j \in \mathbb R^{n_\mu \times n_\mu}$ in step 3 is the sum of all the covariance matrices over all the inputs and outputs. The final covariance matrix $\hat  C$ is the sum of $\hat C^j$ over all the M samples. Consequently, the matrix $U \in \mathbb R^{n_\mu \times r_\mu} (r_\mu \ll n_\mu)$ is computed via the eigendecomposition of $\hat C$ , so that its columns span the active subspace $\mathcal U$ of the parameter domain $\mathcal P$. The active subspace $\mathcal U$ defines the dominant subspace of $\mathcal P$, such that the parameter sampling can be done in $\mathcal U$. Thus, the reduced transfer function can be computed by using $U$ and a projection matrix $V \in \mathbb R^{n\times r}, ~r\ll n$: 
{\small{\begin{align}
    H_r(\mu,s)&=C(UU^T\mu)V(V^TQ(UU^T\mu,s)V)^{-1}V^TB(UU^T\mu).
    \label{eq4}
\end{align}}}

\begin{algorithm}
    \caption{Active subspace method for parametric model order reduction (AS-pMOR)}
    \label{alg:ASpMOR}
    \textbf{Input:} Transfer function of the original system (1) \(H(\mu,s)=C(\mu)Q(\mu,s)^{-1}B(\mu)\in\mathbb{R}^{n_O\times n_I}\),\\
    \hspace{1ex}Active subspace dimension $r_\mu$,\\
    \hspace{1ex}Sampling number $M$ for extracting the active subspace, and $M_V$ for the reduced basis method.\\
    
    \textbf{Output:} Reduced order model $H_r(\mu,s)$.
    
    1. Draw \( M \) samples \( \{(\mu^j,s^j)\}, j=1,\ldots,M, \mu^j \in \mathbb{R}^{n_\mu} \).

    2. Let $H_{o,i}(\mu, s)$ be the $o,i$-th entry of $H(\mu, s)$,
    
    \hspace{3ex}$o=1,\ldots,n_O, i=1,\ldots,n_I$. \\
    \hspace{3ex}For each \( (\mu^j,s^j),j=1,\ldots, M\), compute:

    \hspace{3ex}$\nabla_\mu H_{o,i}(\mu^j, s^j) \in \mathbb R^{n_\mu}$, which is the gradient of $H_{o,i}(\mu,s^j)$ 
    
    \hspace{3ex} w.r.t. $\mu$ and evaluated at $\mu^j$ . 
    
    \hspace{3ex} Let $\mathcal{H}^j \in \mathbb R^{{n_O} \times {n_I} \times {n_\mu}}$ be a 3D tensor, 
    
    \hspace{3ex} its $o,i$-th tube fiber is defined as:
    
    \hspace{3ex} $\mathcal {H}^j(o, i, :):=\mathrm{reshape}(\nabla_\mu H_{o,i}(\mu^j, s^j), 1,1,n_\mu) \in \mathbb R^{1\times 1 \times {n_\mu}}$. 
    
    \hspace{3ex} Here, we use the MATLAB function ``reshape" to write the 
    
    \hspace{3ex} gradient vector into a tensor with dimension $1 \times 1 \times n_\mu$, 
    
    \hspace{3ex} according to the definition of a tensor tube fiber~\cite{KilBHetal13}.

    3. Approximate \( \hat{C} = \frac{1}{M} \sum_{j=1}^{M} \hat{C}^j  \in \mathbb R^{n_\mu \times n_\mu} \),

    \hspace{3ex} \({\hat{C}^j}={\xi^j}^T\cdot \xi^j \),~where $\xi^j:=\mathrm{reshape}(\mathcal H^j,\texttt{[~]},1)\in \mathbb{R}^{n_O n_I\times n_\mu}$.
    
    \hspace{3ex} Here we use the MATLAB function ``reshape" to convert tensor
    
    \hspace{3ex} \(\mathcal H^j\) into a Jacobian matrix. 
    
    4. Compute the eigendecomposition: \( \hat C = W\Lambda W^T \), $W \in \mathbb R^{n_\mu \times n_\mu}$.

    5. $U=W_1, W=[W_1,W_2], W_1 \in \mathbb R^{n_\mu \times r_\mu}$.

    6. Draw $M_V$ samples of ${[\mu_r^k,s^k]},k=1,\ldots,M_V$, \\\hspace{3ex}where $\mu_r=U^T\mu \in \mathbb R^{r_\mu}$.

    7. Compute local projection matrices:
    
    \hspace{3ex}$V_k=Q(U\mu_r^k,s^k)^{-1}B(U{\mu_r}^k), ~k=1,\ldots,M_V$.

    8. Compute global projection matrices: \(V=\mathrm{orth}([V_1,\ldots,V_{M_V}])\).

    9. Construct the reduced order model (ROM):
    
    \hspace{2ex}\small{
    $H_r(\mu,s)=C(UU^T\mu)V(V^TQ(UU^T\mu,s)V)^{-1}V^TB(UU^T\mu)$}.
\end{algorithm}

As the active subspace dimension $r_\mu$ is usually much smaller than the parameter space dimension $n_\mu$, instead of sampling the whole parameter space to get a ROM for $H(\mu,s)$, we only need a much smaller number of samples $M_V$ in the active subspace $\mathcal U$ to get an accurate ROM for $H(U\mu_r,s)$, which is then used as the ROM for $H(\mu, s)$. Compared to sampling $\mu$ in $\mathcal P$, the projection basis $V$ obtained from sampling in the active subspace $\mathcal U$ has a much lower rank $r$, resulting in a ROM with a much smaller size $r$.

Using the final projection matrix $V$ in step 8 of Alg.~\ref{alg:ASpMOR} and the active subspace matrix $U$, we can immediately obtain the AS-pROMs of~(\ref{eq:second_order_pLTI}), (\ref{equ:pLTI}) and (\ref{equ:std_sys}) respectively in the time domain as below:
\begin{align}
    \hspace{-2ex}\begin{array}{rl}
        M_r(\mu_r)\ddot x_r(\mu_r,t)+D_r(\mu_r)\dot x_r(\mu_r,t)+K_r(\mu_r) x_r(\mu_r,t)&=B_r(\mu_r)u(t),\\
        y_r(\mu_r,t)&=C_r(\mu_r)x_r(\mu_r,t),
    \end{array}
    \label{equ:ASpMOR_pLTI_sec_order}
\end{align} where $\mu_r=U^T\mu,~M_r(\mu_r)=V^TM(U\mu_r)V$, $D_r(\mu_r)=V^TD(U\mu_r)V$, $K_r(\mu_r)=V^TK(U\mu_r)V$, $B_r(\mu_r)=V^TB(U\mu_r)$, $C_r(\mu_r)=C(U\mu_r)V$,
\begin{align}
    \begin{array}{rl}
        G_r(\mu_r )\dot x_r(\mu_r,t) &= A_r(\mu_r )x_r(\mu_r,t) + B_r(\mu_r )u(t),\\
        y_r(\mu_r,t) &= C_r(\mu_r )x_r(\mu_r,t),
    \end{array}
    \label{equ:ASpMOR_pLTI}
\end{align}
where $G_r(\mu_r)=V^TG(U\mu_r)V,~A_r(\mu_r)=V^TA(U\mu_r)V$, 
\vspace{-1ex}\begin{align}
    \begin{array}{rl}
         K_r(\mu_r) x_r(\mu_r)&=B_r(\mu_r),  \\
         y(\mu_r)&=C_r(\mu_r)x_r(\mu_r).
    \end{array}
    \label{equ:As_pLIT_std_ROM}
\end{align}

However, it is obvious that a trade-off exists in this algorithm. With the decreasing dimension $r_\mu$ of the active subspace, the error between the ROM and the original model is increased due to the increased error between $H(\mu,s)$ and $H(UU^T\mu,s)$. As a result, we often cannot get a satisfactory balance between the ROM accuracy and the ROM size, especially for systems with large parameter dimensions in real engineering applications. 

In the next section, we propose an iterative active subspace method for pMOR, through which this trade-off is overcome. As a result, much smaller and more accurate ROMs can be derived.

\section{The proposed iterative Active Subspace method}

In this section, the basic framework of our iterative active subspace (IAS) method and its further improvement are proposed in Sec 3.1 and Sec 3.2, respectively. A technique for adaptively choosing the dimensions of the active subspaces and an error indicator are proposed in Sec 3.3. Subsequently, the detailed algorithm with a stopping criterion is presented. Finally, Sec 3.4 proposes a simple post-processing strategy to further reduce the size of the final ROM, and Sec 3.5 introduces an acceleration strategy that makes the offline time of our method comparable to the standard AS method (Alg.~\ref{alg:ASpMOR}).

\subsection{Iterative Active Subspace for parametric Model Order Reduction}

The core concept of the proposed iterative active subspace method involves repeatedly applying the active subspace method to the error system between the original system and the iteratively updated ROM system.

We start by applying the AS and pMOR (ASpMOR) to the transfer function of the original system (1). This process yields an initial ROM of \( H(\mu,s) \), denoted as:
\renewcommand{\arraystretch}{1.5}
\vspace{-1ex}
\begin{align}
    \begin{array}{ll}
        &\mathrm{ASpMOR}_{V_1}^{U_1}(H(\mu,s)):=H_{r_1}(\mu,s)\\
        &=C(U_1U_1^T\mu)V_1(V_1^TQ(U_1U_1^T\mu,s)V_1)^{-1}V_1^TB(U_1U_1^T\mu),
    \end{array}
    \label{single active subspace}
\end{align} 
\renewcommand{\arraystretch}{1}where $\operatorname{colspan}(U_1)$ is the active subspace obtained from the standard AS method Alg.~\ref{alg:ASpMOR}. The matrix $V_1$ is the projection matrix used for pMOR, which is the matrix $V$ computed from step 8 of Alg.~\ref{alg:ASpMOR}.  

After this, we compute the error system \( E_1(\mu,s) = H(\mu,s) - H_{r_1}(\mu,s) \) at the first iteration step. By sampling in the parameter space $\mathcal P$ again, the gradient of \( E_1(\mu,s)\) can be obtained, from which a second active subspace \( U_2 \) is computed, indicating the directions along which the error \( E_1(\mu,s) \) changes fast. Sampling within \( \operatorname{colspan}(U_2) \) allows us to derive the projection matrices for pMOR of both \( H(\mu,s) \) and \( H_{r_1}(\mu,s) \), denoted as \( V_2 \) and \( V_2^1 \), respectively. Consequently, the ROM \( E_{r_1}(\mu,s) \) of the error system \( E_1(\mu,s) \) can be obtained via $U_2,V_2, V_2^1$ as follows:
\renewcommand{\arraystretch}{1.5}
\vspace{-2ex}
\begin{align}
    \begin{array}{rl}
        E_{r_1}(\mu,s)=&\mathrm{ASpMOR}_{V_2}^{U_2}(H(\mu,s))-\mathrm{ASpMOR}_{V_2^1V_1}^{U_2U_1}(H(\mu,s))\\=&C(U_2U_2^T\mu)V_2(V_2^TQ(U_2U_2^T\mu,s)V_2)^{-1}V_2^TB(U_2U_2^T\mu)\\&-C(U_1U_1^TU_2U_2^T\mu)V_1V_2^1({V_2^1}^TV_1^TQ(U_1U_1^TU_2U_2^T\mu,s)V_1V_2^1)^{-1}{V_2^1}^TV_1^TB(U_1U_1^TU_2U_2^T\mu).
    \end{array}
    \label{eq:Er1}
\end{align}     
\renewcommand{\arraystretch}{1}

Note that $V_2^1$ is the projection matrix for pMOR of \( H_{r_1}(\mu,s)\), and \( H_{r_1}(\mu,s)\) is already a reduced-order model which does not need to be reduced further, so that $V_2^1$ can be set as an identity matrix. Thus, $E_{r_1}(\mu,s)$ can be written as:
\renewcommand{\arraystretch}{1.5}
\vspace{-1ex}
\begin{align}
    \begin{array}{rl}
        E_{r_1}(\mu,s)=&\mathrm{ASpMOR}_{V_2}^{U_2}(H(\mu,s))-\mathrm{ASpMOR}_{I_{\text{\hspace{1.05ex} }}V_1}^{U_2U_1}(H(\mu,s))\\=&C(U_2U_2^T\mu)V_2(V_2^TQ(U_2U_2^T\mu,s)V_2)^{-1}V_2^TB(U_2U_2^T\mu)\\\ &-C(U_1U_1^TU_2U_2^T\mu)V_1(V_1^TQ(U_1U_1^TU_2U_2^T\mu,s)V_1)^{-1}V_1^TB(U_1U_1^TU_2U_2^T\mu).
    \end{array}
    \label{eq:Er1_iden}
\end{align}     

\renewcommand{\arraystretch}{1}

Instead of using \( H_{r_1}(\mu,s) \) alone, we approximate the original system \( H(\mu,s) \) as \( H_{r_1}(\mu,s) + E_{r_1}(\mu,s) \). This is denoted as the updated ROM \( H_{r_2}(\mu,s) \) for $H(\mu,s)$.

This process can be iteratively repeated. After obtaining the updated ROM \( H_{r_2}(\mu,s) \), the same steps can be applied to $E_2(\mu,s)=H(\mu,s)-H_{r_2}(\mu,s)$ to further refine the approximation. By continually updating the error system, determining its active subspace, and deriving the corresponding projection matrices, we can iteratively enhance the accuracy of the approximation for the original system \( H(\mu,s) \). Finally, the ROM at the $i$-th iteration step can be represented as:
\renewcommand{\arraystretch}{1.5}\begin{align}
    \begin{array}{ll}
         H_{r_i}(\mu,s)&=H_{r_{i-1}}(\mu,s)+E_{r_{i-1}}(\mu,s),
    \end{array}
    \label{eq:IASpMOR_general}
\end{align} with
\begin{align}
    \begin{array}{ll}
         E_{r_{i-1}}(\mu,s)&=\mathrm{ASpMOR}_{V_i}^{U_i}(H(\mu,s))-\mathrm{ASpMOR}_{I_{\text{\ }}}^{U_i}(H_{r_{i-1}}(\mu,s))\\&=\mathrm{ASpMOR}_{V_i}^{U_i}(H(\mu,s))-H_{r_{i-1}}(U_iU_i^T\mu,s).
    \end{array}
    \label{eq:IASpMOR_general_E}
\end{align}\renewcommand{\arraystretch}{1}Thus, each IAS iteration updates the current ROM by adding a reduced-order approximation of its error system.

Note that approximating $H(\mu, s)$ via iteratively updating its ROM with the ROM of the error system was also proposed in~\cite{morAntBF18}, where no active subspace was considered. With the proposed iterative active subspace method, we show that not only can the projection matrices for pMOR be iteratively constructed, but the active subspace can also be iteratively derived via the error system at each iteration.

Specifically, Table~\ref{table:3_iter} illustrates the expressions of the ROMs generated at the first three iterations.

\begin{small}
   \begin{table}[!htbp]
      \begin{center}
      \renewcommand\arraystretch{1.5}
        \caption{ROMs derived at the first 3 iterations of IAS.}
        \begin{tabular}{cl}
          \hline
          \textbf{\small{Iter}} & \makecell[c]{\textbf{ROMs}}\\
          \hline
          1 &
          \ ${H_{r_1}}(\mu ,s) = {\quad}{\mathrm{ASpMOR}}_{V_1}^{{U_1}}(H(\mu ,s))$\\
          \hline
          2 &
          $\begin{array}{ll}
               {H_{r_2}}(\mu ,s) =& {\mathrm{ASpMOR}}_{V_1}^{{U_1}}(H(\mu ,s)) + {\mathrm{ASpMOR}}_{V_2}^{{U_2}}(H(\mu ,s))\\ &- {\mathrm{ASpMOR}}_{I_{\text{\hspace{1.05ex} }}V_1}^{{U_2}{U_1}}(H(\mu ,s)) 
          \end{array}$\\
          \hline
          3 &
          $\begin{array}{ll}
               {H_{r_3}}(\mu ,s) =& {\mathrm{ASpMOR}}_{V_1}^{{U_1}}(H(\mu ,s)) + {\mathrm{ASpMOR}}_{V_2}^{{U_2}}(H(\mu ,s))\\ &- {\mathrm{ASpMOR}}_{I_{\text{\hspace{1.05ex} }}V_1}^{{U_2}{U_1}}(H(\mu ,s))+{\mathrm{ASpMOR}}_{V_3}^{{U_3}}(H(\mu ,s)) \\&- {\mathrm{ASpMOR}}_{I_{\text{\hspace{1.05ex} }}V_1}^{{U_3}{U_1}}(H(\mu ,s)) - {\mathrm{ASpMOR}}_{I_{\text{\hspace{1.05ex} }}V_2}^{{U_3}{U_2}}(H(\mu ,s))\\&+{\mathrm{ASpMOR}}_{I_{\text{\hspace{1.05ex} }}I_{\text{\hspace{1.05ex} }}V_1}^{{U_3}{U_2}{U_1}}(H(\mu ,s))
          \end{array}$\\
          \hline
        \end{tabular}
        \label{table:3_iter}
      \end{center}\vspace{-3ex}
    \end{table} 
\end{small}
\renewcommand{\arraystretch}{1} 

During this process, we do not have to choose a big active subspace dimension $r_\mu$ at each iteration, which will result in a big ROM size. Instead, only a small $r_\mu$ is needed, and with iteration involved, this process generates a sequence of small sub-ROMs derived from these small active subspaces. For example, at iteration 2 in Table~\ref{table:3_iter}, the single ROM $H_{r_2}(\mu,s)$ is actually composed of 3 sub-ROMs computed from $(U_1,V_1)$, $(U_2,V_2)$, and $(U_1,V_1)$ combined with $(U_2,I)$, respectively. Each of the 3 sub-ROMs is of small size, resulting in a small ROM $H_{r_2}(\mu,s)$. Finally, at the last iteration $N_{iter}$, we obtain the final ROM $H_{r_{N_{iter}}}(\mu,s)$ that is composed of several small sub-ROMs. Computing $H_{r_{N_{iter}}}(\mu,s)$ then reduces to computing the small sub-ROMs. Often, simulating a sequence of small sub-ROMs is cheaper than simulating a big ROM, as the system matrices of the ROMs are dense after projection.
Moreover, the small sub-ROMs generated by this process can be solved in parallel, resulting in even more computational gain.

However, the number of small sub-ROMs generated in this process grows exponentially with the number of iterations. If we can obtain an accurate ROM in a few iterations, the computational gain is still evident. Whereas, if more iterations are required, the exponential increase in the number of sub-ROMs may make this process infeasible.

In the next section, we present an orthogonalization approach that reduces the exponential increase in the number of sub-ROMs to a linear increase, without compromising the accuracy.

\subsection{Orthogonalization between the iterative active subspaces}

The approach proposed in this section is motivated by the following theorems. 
\begin{theorem}
Consider the general IAS process~(\ref{eq:IASpMOR_general}), if \(U_i=U_{i-1}\), s.t. \(V_i=V_{i-1}\), $i=2,\ldots,N_{iter}$, then $E_{r_{i-1}}(\mu,s)=0$.
\label{theo}
\end{theorem}

\noindent The proof of Theorem~\ref{theo} is provided in Appendix~\ref{app:proofs}.

\begin{remark}
From Theorem~\ref{theo}, we see that if $U_i=U_{i-1}$ and $V_i=V_{i-1}$, then the {\it reduced} error system $E_{r_{i-1}}(\mu,s)=0$ and there is no update from $H_{r_{i-1}}(\mu,s)$ to $H_{r_i}(\mu,s)$. If $U_i$ and $V_i$ nearly repeat their previous values, the update may also be small. However, a zero or small update does not imply that the {\it original} error system $E_i(\mu,s)$ is small, especially in the early iterations. Therefore, at each iteration, we should try to make $U_i$ as different from $U_{i-1}$ as possible, such that $V_i$ deviates sufficiently from $V_{i-1}$. This will produce an $E_{r_{i-1}}(\mu, s)$ that approximates $E_{i-1}(\mu, s)$ as well as possible, making the updated $H_{r_i}(\mu, s)=H_{r_{i-1}}(\mu, s)+E_{r_{i-1}}(\mu, s)$ approximate $H(\mu, s)$ as accurately as possible. 
\end{remark}

An optimal way of maximizing the additional information contained in $U_i$ as compared to $U_{i-1}$ is to demand $U_i^TU_{i-1}=0$. Following this rule, we have at the second iteration,
\renewcommand{\arraystretch}{1.5} 
\begin{align}
    \begin{array}{lll}
        {H_{r_2}}(\mu ,s)&={\mathrm{ASpMOR}}_{V_1}^{{U_1}}(H(\mu ,s))+{\mathrm{ASpMOR}}_{V_2}^{{U_2}}(H(\mu ,s))
        \\&\hspace{2.4ex} -{\mathrm{ASpMOR}}_{I_{\text{\hspace{1.05ex} }}V_1}^{{U_2}{U_1}}(H(\mu ,s))
        \\&={\mathrm{ASpMOR}}_{V_1}^{{U_1}}(H(\mu ,s))+{\mathrm{ASpMOR}}_{V_2}^{{U_2}}(H(\mu ,s))\\&\hspace{2.4ex}-H_{r_1}(0,s),
    \end{array}
    \label{eq:iter2_Hr2}
\end{align}\renewcommand{\arraystretch}{1} 
since
\renewcommand{\arraystretch}{1.5}\begin{align}
{\begin{array}{lll}
     &{\mathrm{ASpMOR}}_{I_{\text{\hspace{1.05ex} }}V_1}^{{U_2}{U_1}}(H(\mu ,s))
     \vspace{1ex}\\&=C(U_1U_1^TU_2U_2^T\mu)V_1(V_1^TQ(U_1U_1^TU_2U_2^T\mu,s)V_1)^{-1}V_1^TB(U_1U_1^TU_2U_2^T\mu)
     \vspace{1ex}\\&=C(0)V_1(V_1^TQ(0,s)V_1)^{-1}V_1^TB(0)
     \vspace{1ex}\hspace{24.5ex}(U_1^TU_2=0)\\&=H_{r_1}(0,s) .
\end{array}}
\label{eq:iter2_Hr0}
\end{align}\renewcommand{\arraystretch}{1} 

Let $H_{r_0}(\mu,s)=H_{r_0}(0,s)=H_r(0,s)=\mathrm{ASpMOR}^I_{V_0}(H(0,s))$, where $H_r(0,s)$ is a good ROM approximation for the non-parametrized linear system $H(0,s)$, $V_0$ be its reduced basis, and $U_0=0$. Starting from iteration $i\geq 3$, we have the following theorem.

\begin{theorem}
Assume that \(\forall\text{ } l,\text{ } j \leq i-1\hspace{1ex}(i\geq 3)\) with \(l \neq j\), \(U_l^T U_j = 0\). If \(\exists\text{ } g < i-1\), such that \(U_i = U_g\) and \(V_i = V_g\), then \(H_{r_i}(\mu, s) = H_{r_{i-1}}(\mu, s)-H_{r_{i-1}}(0,s)+H_{r_i}(0,s)\).
\label{theo_all}
\end{theorem}

\noindent The proof of Theorem~\ref{theo_all} is provided in Appendix~\ref{app:proofs}.

Thus, if \(\exists\text{ } g < i-1\), \(U_i = U_g\), s.t. \(V_i = V_g\), then $H_{r_i}(\mu, s)=H_{r_{i-1}}(\mu,s)-H_{r_{i-1}}(0,s)+H_{r_i}(0,s)$. This implies that $E_{r_{i-1}}(\mu,s)=-H_{r_{i-1}}(0,s)+H_{r_i}(0,s)$ becomes constant with respect to the parameter $\mu$ (see (\ref{eq:IASpMOR_general})).

Moreover, if the snapshots at the samples $(\mu=0, s_j), j=1,\ldots,m,$ are included to construct $V_i, i=1,\ldots, N_{iter},$ at each iteration, then $H_{r_i}(\mu,s), i=1,\ldots, N_{iter},$ has similar accuracy at $\mu=0$, since all $H_{r_i}(\mu,s), i=1,\ldots, N_{iter},$ interpolate the original transfer function $H(\mu, s)$ at the same samples $(\mu=0, s_j), j=1,\ldots, m$, i.e., 
$$H(0, s_j)=H_{r_i}(0, s_j), i=1,\ldots, N_{iter}, j=1, \ldots, m.$$ 

Once more, they will have similar accuracy as a ROM of $H(0,s)$ obtained from the same frequency samples $s_j, j=1,\ldots,m$, since the same interpolation conditions are satisfied, i.e., 

\begin{align}
    H(0, s_j)=H_{r_i}(0, s_j)=H_r(0, s_j),~i=1,\ldots, N_{iter},~j=1, \ldots, m.
    \label{eq:H_r_0}
\end{align}

As a result, we have $E_{r_{i-1}}(\mu,s)=-H_{r_{i-1}}(0,s)+H_{r_i}(0,s)\approx 0$, and there will be no (or ignorable) updates from $H_{r_{i-1}}(\mu,s)$ to $H_{r_{i}}(\mu,s)$. Therefore, at each iteration ($i \geq 3$), we should further try to make $U_i$ as different from all $U_{j}$ ( $j<i$) as possible, such that $V_i$ is largely different from all $V_{j}$. This may produce an $E_{r_{i-1}}(\mu, s)$ that approximates $E_{i-1}(\mu, s)$ as well as possible.

An optimal way of making $U_i$ as different from all $U_{j}$ ($j<i$) as possible is to have \(U_i\) further orthogonalized against the columns in \(U_1,\ldots,U_{i-1}\), i.e. \({U_i}^TU_j=0,j=1,\ldots,i-1\), via, e.g., the modified Gram-Schmidt process (MGS). Consequently, we can get a much more concise expression of \(H_{r_i}(\mu,s)\) from (\ref{eq:accumulative_IAS_pMOR_half}):
\renewcommand{\arraystretch}{1.5}\begin{align}
    \begin{array}{ll}
         H_{r_i}(\mu,s)&=H_{r_{i-1}}(\mu,s)+\mathrm{ASpMOR}_{V_{i}}^{U_{i}}(H(\mu,s))-H_{r_{i-1}}(0,s)\\&=H_{r_0}(\mu,s)+\sum\limits_{j=1}\limits^{i}\mathrm{ASpMOR}_{V_{j}}^{U_{j}}(H(\mu,s))-\sum\limits_{j=0}\limits^{i-1} H_{r_j}(0,s),~i\geq1.
    \end{array}
    \label{eq:IASpMOR_general_orth}
\end{align}\renewcommand{\arraystretch}{1}

Furthermore, according to (\ref{eq:H_r_0}), we can approximate each $H_{r_{i}}(0,s)$, $i=1,2,\ldots$, as $H_{r}(0,s)$. Then (\ref{eq:IASpMOR_general_orth}) can be simplified to

\renewcommand{\arraystretch}{1.5}\begin{align}
    \begin{array}{ll}
         H_{r_i}(\mu,s)&=H_{r_{i-1}}(\mu,s)+\mathrm{ASpMOR}_{V_{i}}^{U_{i}}(H(\mu,s))-H_{r_{i-1}}(0,s)\\&\approx
       H_{r_{i-1}}(\mu,s)+\mathrm{ASpMOR}_{V_i}^{U_i}(H(\mu,s))-H_{r}(0,s) \\&=\sum\limits_{j=1}\limits^{i}\mathrm{ASpMOR}_{V_{j}}^{U_{j}}(H(\mu,s))-(i-1)\cdot H_{r}(0,s), ~~~~~~~~~~~i\geq1.
    \end{array}
    \label{eq:IASpMOR_general_orth_hr}
\end{align}\renewcommand{\arraystretch}{1}

Since $H(0,s)$ is a linear system without parameters, this ROM $H_r(0,s)$ can be easily obtained by model order reduction methods, such as the multi-moment-matching method~\cite{morFenAB17}, iterative rational Krylov methods~\cite{morGugAB08}, the structure-preserving interpolatory method for second-order systems~\cite{morBeaG09}, the proper orthogonal decomposition (POD) method~\cite{morBenGQetal21a}, etc. Without sampling in the parameter domain, the size of $H_r(0,s)$ can be much smaller than the sizes of the sub-ROMs $\mathrm{ASpMOR}_{V_j}^{U_j}(H(\mu,s))$ $(j=1,\ldots,i)$ in the expression of $H_{r_i}(\mu,s)$. Thus, variations in the size of $H_r(0,s)$ resulting from the choice of reduction method have only a minor effect on the overall ROM size. In this work, we use the POD method to construct $H_r(0,s)$.

In each iteration $i$, a new sub-ROM $\mathrm{ASpMOR}_{V_i}^{U_i}(H(\mu,s))$ is added to $H_{r_{i-1}}(\mu,s)$. The time domain expression of $\mathrm{ASpMOR}_{V_i}^{U_i}(H(\mu,s))$ is similar to (\ref{equ:ASpMOR_pLTI_sec_order}), (\ref{equ:ASpMOR_pLTI}) and (\ref{equ:As_pLIT_std_ROM}), which is
\begin{align}
    \begin{array}{rl}
        M_{r_i}(\mu_{r_i})\ddot x_{r_i}(\mu_{r_i},t)+D_{r_i}(\mu_{r_i})\dot x_{r_i}(\mu_{r_i},t)+K_{r_i}(\mu_{r_i}) x_{r_i}(\mu_{r_i},t)&=B_{r_i}(\mu_{r_i})u(t),\\
        y_{r_i}(\mu_{r_i},t)&=C_{r_i}(\mu_{r_i})x_{r_i}(\mu_{r_i},t),
    \end{array}
    \label{equ:sub_IASpMOR_pLTI_sec_order}
\end{align} for the second-order system (\ref{eq:second_order_pLTI}), where $\mu_{r_i}={U_i}^T\mu,~M_{r_i}(\mu_{r_i})=V_i^TM({U_i}\mu_{r_i})V_i$, $D_{r_i}(\mu_{r_i})=V_i^TD({U_i}\mu_{r_i})V_i$, $K_{r_i}(\mu_{r_i})=V_i^TK({U_i}\mu_{r_i})V_i$, $B_{r_i}(\mu_{r_i})=V_i^TB({U_i}\mu_{r_i})$, $C_{r_i}(\mu_{r_i})=C({U_i}\mu_{r_i})V_i$, or
\begin{align}
    \begin{array}{rl}
        G_{r_i}(\mu_{r_i})\dot x_{r_i}(\mu_{r_i},t) &= A_{r_i}(\mu_{r_i} )x_{r_i}(\mu_{r_i},t) + B_{r_i}(\mu_{r_i} )u(t),\\
        y_{r_i}(\mu_{r_i},t) &= C_{r_i}(\mu_{r_i})x_{r_i}(\mu_{r_i},t),
    \end{array}
    \label{equ:sub_IASpMOR_pLTI}
\end{align}
for the first-order system (\ref{equ:pLTI}), where $G_{r_i}(\mu_{r_i})=V_i^TG({U_i}\mu_{r_i})V_i$, $A_{r_i}(\mu_{r_i})=V_i^TA({U_i}\mu_{r_i})V_i$, or
\begin{align}
    \begin{array}{rl}
         K_{r_i}(\mu_{r_i}) x_{r_i}(\mu_{r_i})&=B_{r_i}(\mu_{r_i}),  \\
         y(\mu_{r_i})&=C_{r_i}(\mu_{r_i})x_{r_i}(\mu_{r_i}),
    \end{array}
    \label{equ:IAs_pLIT_std_ROM}
\end{align}
for the static system (\ref{equ:std_sys}).

According to (\ref{eq:IASpMOR_general_orth}) and (\ref{eq:IASpMOR_general_orth_hr}), the final estimated solution $\hat y_i(\mu,t)$ (or $\hat y_i(\mu)$)
at the $i$-th iteration of our method, can be derived as
\begin{align}
    \hat y_i(\mu,t)&=\sum_{j=1}\limits^{i}y_{r_j}(\mu_{r_j},t)-\sum_{j=1}\limits^{i-1} y_{r_j}(0,t)\\&\approx\sum_{j=1}\limits^{i}y_{r_j}(\mu_{r_j},t)-(i-1)\cdot y_r(0,t),
\end{align}
or
\begin{align}
    \hat y_i(\mu)&=\sum_{j=1}\limits^{i}y_{r_j}(\mu_{r_j})-\sum_{j=1}\limits^{i-1} y_{r_j}(0)\\&\approx\sum_{j=1}\limits^{i}y_{r_j}(\mu_{r_j})-(i-1)\cdot y_r(0),
\end{align}
where $y_r(0,t)$ and $y_r(0)$ are the time domain dynamic and static solutions of $H_r(0,s)$, respectively.

After orthogonalization (\ref{eq:IASpMOR_general_orth}) and further simplification (\ref{eq:IASpMOR_general_orth_hr}), the reduced transfer functions in Table~\ref{table:3_iter} can be rewritten into those in Table~\ref{table:3_iter_orth}. With orthogonalization between the matrices of active subspaces and reasonable approximation, only one new sub-ROM is generated at each iteration. Thus, the effectiveness of this iterative active subspace method is greatly improved when many iterations are needed to achieve a satisfactory accuracy.

\renewcommand{\arraystretch}{1}
\begin{small}
   \begin{table}[!htbp]
      \begin{center}
      \renewcommand\arraystretch{1.5}
        \caption{ROMs at the first 3 iterations of IAS with orthogonalization between $U_i$ }
        \begin{tabular}{cl}
          \hline
          \textbf{\small{Iter}} & \makecell[c]{\textbf{ROMs}}\\
          \hline
          1 &
          \ ${H_{r_1}}(\mu ,s) ={\quad}{\mathrm{ASpMOR}}_{V_1}^{{U_1}}(H(\mu ,s))$\\
          \hline
          2 &
          $\begin{array}{ll}
               {H_{r_2}}(\mu ,s) =& {\mathrm{ASpMOR}}_{V_1}^{{U_1}}(H(\mu ,s)) + {\mathrm{ASpMOR}}_{V_2}^{{U_2}}(H(\mu ,s))\\ &- H_r(0,s) 
          \end{array}$\\
          \hline
          3 &
          $\begin{array}{ll}
               {H_{r_3}}(\mu ,s) =& {\mathrm{ASpMOR}}_{V_1}^{{U_1}}(H(\mu ,s)) + {\mathrm{ASpMOR}}_{V_2}^{{U_2}}(H(\mu ,s))\\ &+{\mathrm{ASpMOR}}_{V_3}^{{U_3}}(H(\mu ,s)) \\&- 2H_r(0,s)
          \end{array}$\\
          \hline
        \end{tabular}\vspace{-4ex}
        \label{table:3_iter_orth}
      \end{center}
    \end{table} 
\end{small}
\renewcommand{\arraystretch}{1}

\subsection{Adaptive active subspace dimension decision and stopping criteria}

To make the algorithm more flexible, we choose the active subspace dimension
$r_\mu$ adaptively at each iteration. Let
$\lambda_1\geq\lambda_2\geq\cdots\geq\lambda_{n_\mu}\geq0$
denote the eigenvalues, in descending order, of the covariance matrix $\hat C$
constructed from the gradients of the current error system. Given a prescribed
energy ratio $0<\alpha_\mu<1$, $r_\mu$ is selected such that
\begin{align}
    \frac{\sum_{l=1}^{r_\mu}\lambda_l}
    {\sum_{k=1}^{n_\mu}\lambda_k}
    <\alpha_\mu<
    \frac{\sum_{l=1}^{r_\mu+1}\lambda_l}
    {\sum_{k=1}^{n_\mu}\lambda_k}.
    \label{eq:energy_crit}
\end{align}

Furthermore, at the i-th iteration, since $H(\mu, s)-H_{r_i}(\mu, s)=E_i(\mu, s)\approx E_{r_i}(\mu, s)$, when $E_{r_i}(\mu, s)$ approximates $E_i(\mu, s)$ sufficiently well, we can use $E_{r_i}(\mu,s)$ as an estimator for the error between $H_{r_i}(\mu,s)$ and $H(\mu,s)$, which can also be used as the stopping criterion of the proposed pMOR algorithm with IAS.

The overall algorithm of our iterative active subspace approach for pMOR is shown in Algorithm~\ref{alg:IASpMOR}.

\subsection{Post processing}

In Algorithm~\ref{alg:IASpMOR}, the size of the ROMs generated in each iteration, denoted as $r_i$, satisfies
\begin{equation*}
    r_i = \text{rank}(\text{orth}([V_{1,i},\ldots,V_{M_V,i}]))\leq M_V.
\end{equation*}
To ensure stable error reduction during iterations, we recommend selecting $M_V$ as a relatively large value, e.g., adaptively set $M_V= \alpha \cdot{r_{\mu,i}}^2$ in each iteration ($\alpha $ is a user-defined scaling constant). After achieving satisfactory accuracy through $N_{iter}$ iterations of the IAS-pMOR process, we can effectively truncate the reduced basis $V_i, i=1,..,N_{iter},$ in a simple way while maintaining accuracy (with active subspaces $U_1$ to $U_{N_{iter}}$ remaining unchanged). 

This truncation is feasible because the intentionally oversized $M_V$ during iterations ensures sufficient ROM size to maintain minimal approximation error between ${\mathrm{ASpMOR}}_{V_i}^{{U_i}}(H(\mu ,s))$ and $H(U_iU_i^T\mu,s)$, which is crucial for computing the next $U_{i+1}$. Specifically, in Step 10 of Algorithm~\ref{alg:IASpMOR}, we perform singular value decomposition (SVD) for the orthogonalization process (\(V_i=\mathrm{orth}([V_{1,i},\ldots,V_{M_V,i}])\)): $V_i\Sigma_iR_i = \text{svd}([V_{1,i},\ldots,V_{M_V,i}])$. During post-processing, we can simply retain only the first $\beta\cdot r_{\mu,i}$ columns of each $V_i$ ($\beta $ is a user-defined constant), so that $r_i=\beta\cdot r_{\mu,i}$. Experimental results demonstrate that this column truncation strategy significantly reduces the ROMs size with negligible accuracy degradation.

\begin{small}

\begin{algorithm}[!htbp]
\caption{Iterative Active Subspace approach for parametric Model Order Reduction (IAS-pMOR)}\label{alg:IASpMOR}
\textbf{Data:} Transfer function of the original system \\\hspace{7.5ex}\(H(\mu,s)=C(\mu,s)Q(\mu,s)^{-1}B(\mu,s)\in\mathbb{R}^{n_O\times n_I}\),\\
    \hspace{7.5ex}Energy ratio $\alpha_\mu$ for active subspace dimension selection,\\
    \hspace{7.5ex}Sampling number $M$ for active subspace extraction,\\
    \hspace{7.5ex}$tol<1$ the error tolerance.\\
\textbf{Result:} Reduced-order model $H_r(\mu,s)$.\\
\textbf{Initialization:} Iteration number $i=0$, \(error=Inf\),\\
\hspace{7.2em}Compute $H_r(0,s)=\text{POD}(H(0,s))$,\\
 \hspace{7.2em}$H_{r_0}(\mu,s)=H_r(0,s)$, $U_0=0$,\\
\hspace{7.2em}Draw \( M \) samples \( \{(\mu^j,s^j)\}, j=1,\ldots,M, \mu^j \in \mathbb{R}^{n_\mu}. \)\\
\While{$error>tol$}{

    1. Iterate: ${i}={i}+1$,
     
     \hspace{11ex}$E_{{i-1}}(\mu,s)=H(\mu,s)-H_{r_{i-1}}(\mu,s)$,~ $E(\mu,s)=E_{i-1}(\mu,s)$.

    2. Let $E_{o,\iota}(\mu, s)$ be the $o,\iota$-th entry of $E(\mu, s)$,
    
    \hspace{3ex}$o=1,\ldots,n_O, \iota=1,\ldots,n_I$.
    
    \hspace{3ex}For each \( (\mu^j,s^j),j=1,\ldots, M\), compute:

    \hspace{3ex}$\nabla_\mu E_{o,\iota}(\mu^j, s^j) \in \mathbb R^{n_\mu}$, which is the gradient of $E_{o,\iota}(\mu,s)$ 
    
    \hspace{3ex} w.r.t. $\mu$ and evaluated at $(\mu^j,s^j)$ . 
    
    \hspace{3ex} Let $\mathcal{E}^j \in \mathbb R^{{n_O} \times {n_I} \times {n_\mu}}$ be a 3D tensor, its $o,\iota$-th tube fiber is:
    
    \hspace{3ex} $\mathcal {E}^j(o, \iota, :):=\mathrm{reshape}(\nabla_\mu E_{o,\iota}(\mu^j, s^j), 1,1,n_\mu) \in \mathbb R^{1\times 1 \times {n_\mu}}$.

    3. Approximate \( \hat{C} = \frac{1}{M} \sum_{j=1}^{M} \hat{C}^j  \in \mathbb R^{n_\mu \times n_\mu} \),

    \hspace{3ex} \({\hat{C}^j}={\xi^j}^T\cdot \xi^j \),~where $\xi^j:=\mathrm{reshape}(\mathcal E^j,\texttt{[~]},1)\in \mathbb{R}^{n_O\cdot n_I\times n_\mu}$.
    
    4. Compute the eigendecomposition \( \hat C = W\Lambda W^T \), $W \in \mathbb R^{n_\mu \times n_\mu}$.

    5. Choose $r_{\mu,i}$, $\frac{\sum_{l=1}^{r_{\mu,i}}\lambda_l}{\sum_{k=1}^{n_\mu}\lambda_k}<\alpha_\mu<\frac{\sum_{l=1}^{r_{\mu,i}+1}\lambda_l}{\sum_{k=1}^{n_\mu}\lambda_k}$, 
    
    \hspace{3ex}$\Lambda=\operatorname{diag}(\lambda_1,\ldots,\lambda_{n_\mu}),~\lambda_1\geq\lambda_2\geq\ldots\geq\lambda_{n_\mu}$.

    6. $\hat{U}_{i}=W_1, W=[W_1,W_2], W_1 \in \mathbb R^{n_\mu \times r_{\mu,i}}$.

    7. Orthogonalize \(\hat{U_i}\) against $[U_0,\ldots,U_{i-1}]$ to get $U_i$, via, e.g., MGS.

    8. Draw $M_V$ (e.g., $M_V={r_{\mu,i}}^2$) samples of $(\mu_{r_i}^k,s^k)$,
    
    \hspace{3ex}$k=1,\ldots,M_V$, $ \mu_{r_i}^k\in \mathbb{R}^{r_{\mu,i}}$. 

    9. Compute the local reduced basis at the $i$-th iteration: 
    
    \hspace{3ex}\(V_{k,i}=(s^kG(U_i\mu_{r_i}^k)-A(U_i\mu_{r_i}^k))^{-1}B(U_i\mu_{r_i}^k)\).

    \hspace{-1ex}10. Compute the global projection matrix: \(V_i=\mathrm{orth}([V_{1,i},\ldots,V_{M_V,i}])\).

    \hspace{-1ex}11. Construct the ROM:
    
    \vspace{0.5ex}\begin{small}
        \(
        \hspace{2ex}\begin{array}{lll}
            H_{r_{i}}(\mu,s)&= E_{r_{i-1}}(\mu,s)+H_{r_{i-1}}(\mu,s)\\
            &=C(U_{i}{U_{i}}^T\mu)V_{i}(V_{i}^TQ(U_{i}{U_{i}}^T\mu,s)V_{i})^{-1}V_{i}^TB(U_{i}{U_{i}}^T\mu)\\
            &~~~-H_{r_{i-1}}(0,s)+H_{r_{{i-1}}}(\mu,s).
        \end{array}
        \)
    \end{small}

    \hspace{-1ex}12. Estimate the ROM error:
    
    \hspace{3ex}Draw L samples  \( \{(\mu_l,s_l)\}, l=1,\ldots,L, \mu_l \in \mathbb{R}^{n_\mu} \),
    
    \hspace{3.25ex}\(error=\frac{\Sigma_l|| E_{r_{i-1}}(\mu_l,s_l)||_F}{\Sigma_l||H_{r_{i}}(\mu_l,s_l)||_F}\).
}
\end{algorithm}
    
\end{small}

\subsection{IAS with accelerated gradient computation} \label{subsec:gradient_acceleration}
The computational bottleneck of the IAS algorithm (Algorithm~\ref{alg:IASpMOR}) lies in Step~2, where gradients must be computed repeatedly for all entries of the updated error function \( E_i(\mu, s) \) in every iteration. To clarify this process, let \( e_i(\mu, s) \) denote an arbitrary entry of \( E_i(\mu, s) \). Similarly, define the corresponding entries of the reduced-order error \( E_{r_i}(\mu, s) \), full-model function \( H(\mu, s) \), and reduced-model function \( H_{r_i}(\mu, s) \) as \( e_{r_i}(\mu, s) \), \( h(\mu, s) \), and \( h_{r_i}(\mu, s) \), respectively.  We propose an efficient recursive computation strategy that significantly reduces this overhead through the following insight:
\begin{align}
    \nabla_\mu e_i(\mu,s) = 
    \begin{cases}
        \nabla_\mu h(\mu,s), & i = 0, \\
         \nabla_\mu h(\mu ,s) - \nabla_\mu {h_{{r_i}}}(\mu ,s), & i \geq 1.
    \end{cases}
    \label{eq:recursive_gradient}
\end{align}

This relationship reveals that after the initial iteration, we only need to compute the gradient of the reduced-order model (ROM) $\nabla h_{r_{i}}(\mu,s)$. Since $h_{r_{i}}(\mu,s)$ operates in a reduced-order space, its gradient computation becomes substantially cheaper compared to evaluating the gradient $\nabla_\mu h(\mu,s)$ of the FOM in (\ref{eq:recursive_gradient}).

Moreover, from (\ref{eq:IASpMOR_general}), we have that for $i\geq1$,
\begin{align}
\hspace{-2ex}\begin{array}{ll}
     \nabla_\mu e_i(\mu,s) &=\nabla_\mu h(\mu ,s) - \nabla_\mu {h_{{r_i}}}(\mu ,s) \\&= \nabla_\mu h(\mu ,s) - (\nabla_\mu {e_{{r_{i - 1}}}}(\mu ,s) + \nabla_\mu {h_{{r_{i - 1}}}}(\mu ,s))\\
     &= \nabla_\mu {e_{i - 1}}(\mu ,s) - \nabla_\mu {e_{{r_{i - 1}}}}(\mu ,s).
\end{array}
\label{eq:lazy_gradient}
\end{align}

According to this recursive relationship, instead of computing $\nabla h_{r_{i}}(\mu,s)$, we only need to compute $\nabla e_{r_{i-1}}(\mu,s)$. The expression in (\ref{eq:IASpMOR_general_orth}) shows that $H_{r_i}(\mu,s)=H_{r_0}(\mu,s)+\sum\limits_{j=1}\limits^{i}\mathrm{ASpMOR}_{V_{j}}^{U_{j}}(H(\mu,s))-\sum\limits_{j=0}\limits^{i-1} H_{r_j}(0,s)$. Thus, in the $i-$th iteration, computing the gradient $h_{r_i}(\mu,s)$ includes computing the gradient of $i$ sub-ROMs, while $E_{r_{i-1}}(\mu,s)=H_{r_i}(\mu,s)-H_{r_{i-1}}(\mu,s)=\mathrm{ASpMOR}_{V_{i}}^{U_{i}}(H(\mu,s))-H_{r_{i-1}}(0,s)$ reveals that computing $\nabla e_{r_{i-1}}(\mu,s)$ only requires computing the gradient of one sub-ROM. As a result, computing $\nabla e_{r_{i-1}}(\mu,s)$ is much cheaper than computing $\nabla h_{r_{i}}(\mu,s)$, especially when the iteration number $i$ is large.

Through this optimization, the offline computation time for IAS becomes comparable to that of the standard AS algorithm (Algorithm~\ref{alg:ASpMOR}), while retaining its enhanced approximation capability.

\FloatBarrier
\section{Numerical experiments}

In this section, the proposed iterative active subspace method is validated using two mechanical models with numerous parameters.

To evaluate the accuracy of parametric reduced-order models, we adopt the following error metric:
\begin{align}
    \varepsilon &= \frac{\sum_{l=1}^m \|\hat{X}(\mu^l) - X(\mu^l)\|_F}{\sum_{l=1}^m \|X(0) - X(\mu^l)\|_F},
\label{equ:error}
\end{align}
where \( m \) denotes the number of parameter samples and \( \mu^l \) represents a sampled parameter. For time-dependent problems, \( \hat{X}(\mu)=[\hat x(\mu,t_1),\ldots,\) \(\hat x(\mu,t_{\tau})]\in \mathbb{R}^{n\times\tau} \) includes the approximate solutions computed from the ROM and \( X(\mu)=[x(\mu,t_1),\ldots,x(\mu,t_{\tau})]\in \mathbb{R}^{n\times\tau} \) denotes the full-order model (FOM) solutions (both over the same time grid \([t_1,\ldots,t_\tau]\)). For steady-state problems, \( \hat{X}(\mu)\) is the approximate solution obtained from the ROM and \( X(\mu)=x(\mu)\) in~(\ref{equ:std_sys}). The ROM here may refer to either of the three cases: the single ROM from pMOR without AS, from AS-pMOR (Alg.~\ref{alg:ASpMOR}) , or the IAS (Alg.~\ref{alg:IASpMOR}) produced ROM. \( X(0 )\) specifies the FOM reference solution at a nominal parameter value (e.g., \( \mu=0 \)).

Unlike the conventional relative error metrics that normalize the absolute error by the magnitude of $X(\mu)$, our criterion quantifies the error relative to the deviation of the solution $X(\mu)$ from the baseline solution $X(0)$. This design avoids undervaluing the error when \( X(\mu^l) \) exhibits small variations compared to its mean over the $m$ samples, thereby providing a stricter and more physically meaningful measure of the ROM error.

\subsection{Magnetic Actuator}
MEMS electromagnetic actuators are a common type of MEMS actuator, mainly used in Micro-Opto-Electro-Mechanical Systems (MOEMS), such as super-resolution imaging, Lidar, etc. They have the advantages of small size, light weight, and easy integration. Figure~\ref{fig:MEMS_act_para} illustrates the structural layout of a typical 2D in-plane electromagnetic actuator (developed at Tsinghua University~\cite{WanFLetal24}), on which we defined 25 geometric parameters $\mu_{1}{\sim}\mu_{25}$, covering almost all the dimensions of the beams and mass blocks.

\begin{figure}[!htbp]
    \centering
    \includegraphics[width=0.6\linewidth]{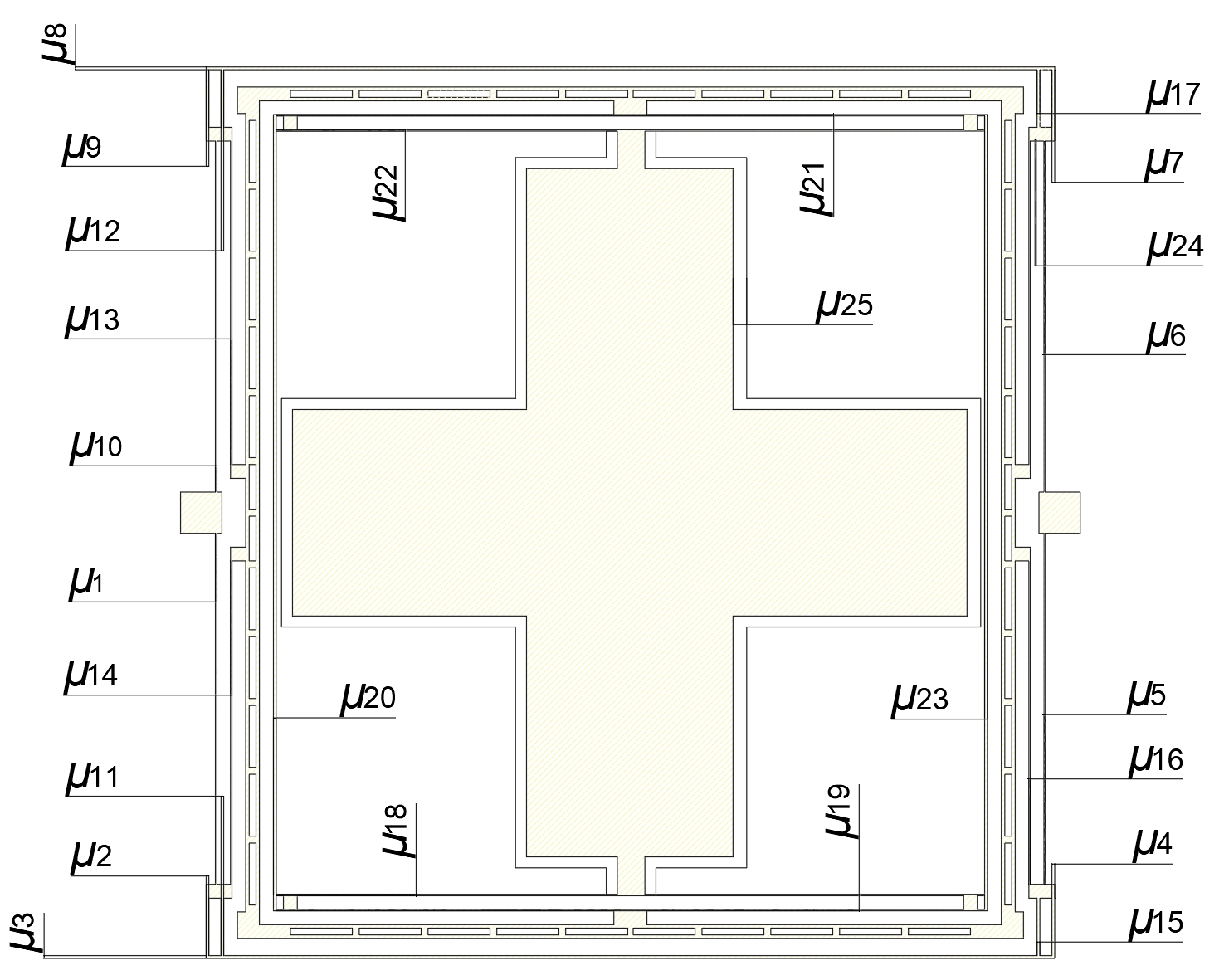}
    \caption{2D model of the magnetic actuator defining 25 parameters $\mu_{1}{\sim}\mu_{25}$.}
    \label{fig:MEMS_act_para}
\end{figure}

We aim at modeling the MEMS actuator performance deviations using the proposed iterative active subspace pMOR method (Alg.~\ref{alg:IASpMOR}). In particular, we are interested in the static displacement of all the mesh points (steady-state solution) under the same magnetic force corresponding to different parameter variances in the 25-dimensional parameter space~\cite{WanFLetal24}. 

The static displacement $x(\mu)$ of this structure can be derived by solving the steady parametric system in~(\ref{equ:std_sys}) with $n=21,914$, $ x(\mu) \in \mathbb R^{n}$, $n_\mu=25$, $\mu=[\mu_1,\mu_2,\ldots,\mu_{n_\mu}]^T\in \mathbb R^{n_\mu}$.

The derivatives of the solution $x(\mu)$, w.r.t. the $k$-th parameter $\mu_k$ of $\mu$ can be written as:
\begin{align}
    \begin{array}{ll}
         &\frac{\partial x(\mu)}{\partial \mu_k}=K(\mu)^{-1}\frac{\partial B(\mu)}{\partial \mu_k}-K(\mu)^{-1}\frac{\partial K(\mu)}{\partial \mu_k}K(\mu)^{-1}B(\mu)\in \mathbb R^{n},
    \end{array}
    \label{eq:derivative_static}
\end{align}
such that its derivatives w.r.t. all the parameters $\mu_k, k=1,\ldots, n_\mu=25,$ can be written into a Jacobian matrix:
\begin{align}
         \nabla_{\mu}x(\mu)=[\frac{\partial x(\mu)}{\partial \mu_1},\frac{\partial x(\mu)}{\partial \mu_2},\ldots,\frac{\partial x(\mu)}{\partial \mu_{25}}] \in \mathbb R^{n\times n_\mu}.
         \label{eq:Jacobian_static}
\end{align}
$M=50$ sampling points $\mu^1, \ldots, \mu^M$ are chosen for computing the active subspace. It is not difficult to verify that $\hat C$ in step 3 of Alg.~\ref{alg:ASpMOR} can be computed as:
\begin{align}
    \hat{C} = \frac{1}{M} \sum_{j=1}^{M} (\nabla_{\mu}x(\mu^j))^T(\nabla_{\mu}x(\mu^j)) \in \mathbb{R}^{n_\mu \times n_\mu}.
\end{align}

For step 3 in Alg.~\ref{alg:IASpMOR}, $E_i(\mu)=x(\mu)-\hat x(\mu)$, where $\hat x(\mu)$ is the approximate solution at the $i$-th iteration, and
\begin{align}
    \nabla_\mu E_i(\mu)=[\frac{\partial x(\mu)}{\partial \mu_1}-\frac{\partial \hat x(\mu)}{\partial \mu_1},\frac{\partial x(\mu)}{\partial \mu_2}-\frac{\partial \hat x(\mu)}{\partial \mu_2},\ldots,\frac{\partial x(\mu)}{\partial \mu_{25}}-\frac{\partial \hat x(\mu)}{\partial \mu_{25}}] \in \mathbb R^{n\times n_\mu}.
\end{align}
We can compute $\hat C$ in a similar way as
\begin{align}
    \hat{C} = \frac{1}{M} \sum_{j=1}^{M} (\nabla_{\mu}E_i(\mu^j))^T(\nabla_{\mu}E_i(\mu^j))\in \mathbb{R}^{n_\mu \times n_\mu}.
\end{align}

Within iterations, we apply the proposed adaptive active subspace technique from Section 3.3 to adaptively decide the active subspace dimension $r_{\mu,i}$ at each iteration. Here, the energy ratio is set as $\alpha_\mu=0.9$. The projection basis $V_i$ (step 9 in Alg.~\ref{alg:IASpMOR}) is derived with $M_{V} = r_{\mu,i}^2$. For this example, the column-space dimension of $V_i$ is further truncated to $r_i = 2 \cdot r_{\mu,i}$ in the post-processing phase for Alg.~\ref{alg:IASpMOR}; the resulting accuracy is reported below.

The standard active subspace method (Alg.~\ref{alg:ASpMOR}) allows parametric exploration through different combinations of the active subspace dimension \(r_\mu\) and ROM size \(r\). Figure~\ref{fig:magact_25_UVError} illustrates the error distribution over different \((r_\mu,r)\) combinations. We can see that when a smaller active subspace dimension $r_\mu$ is selected, the error rapidly decreases as the ROM size $r$ increases, but quickly plateaus and shows no further reduction even with significantly larger $r$. In contrast, choosing a larger active subspace dimension $r_\mu$ leads to a slower error decay with increasing ROM size $r$, yet ultimately achieves a lower error due to the enhanced representation capacity of the subspace. For each fixed ROM size \(r\), the approximation error exhibits a characteristic pattern when varying \(r_\mu\): it initially decreases as \(r_\mu\) increases, attaining a minimum error at a specific \(r_\mu\), and then rises with further increases of \(r_\mu\). The dashed line in Figure~\ref{fig:magact_25_UVError} indicates the minimal errors produced by the active subspaces corresponding to every fixed ROM size \(r\). 

As the system matrices of ROMs are usually dense matrices, the computational complexity increases cubically with ROM size. Therefore, solving $N_{ROM}$ ROMs with size $r_l,~l=1,\ldots,N_{ROM},$ is equivalent to solving one ROM with an equivalent size $r_{eq}=(\sum_{l=1}^{N_{ROM}}(r_l^3))^{1/3}$.

From Figure~\ref{fig:ias_decrease_25} (right), we can see that while the standard active subspace method (Alg.~\ref{alg:ASpMOR}) shows great advantage compared to the snapshot method (mathematically equivalent to the standard active subspace method with \(r_\mu = n_\mu\)) where the ROM sizes $r$ are small, this advantage diminishes with the increment of $r$.

Instead, our iterative active subspace method (Alg.~\ref{alg:IASpMOR}) achieves the same precision with an equivalent ROM size $r_{eq}$ much smaller than the size of the ROM computed from the standard active subspace method with a single active subspace. Moreover, the small ROMs obtained through the IAS method can be solved in parallel, leading to even greater computational gains. 

Table~\ref{table:mag_static} presents the computation time for computing the steady-state solution of the ROMs obtained through the AS and IAS methods at a randomly selected parameter sample $\mu^*=[\mu_1^*, \ldots, \mu_{25}^*]^T$. When the ROMs from the IAS method are solved in parallel, the computation time is significantly reduced compared to the AS method. Compared to the runtime of directly simulating the original steady system with a sparse solver, such as the preconditioned conjugate gradient method (PCG), the IAS method is even faster. 

\FloatBarrier
\begin{figure}[!htbp]
    \centering
    \includegraphics[width=0.68\linewidth]{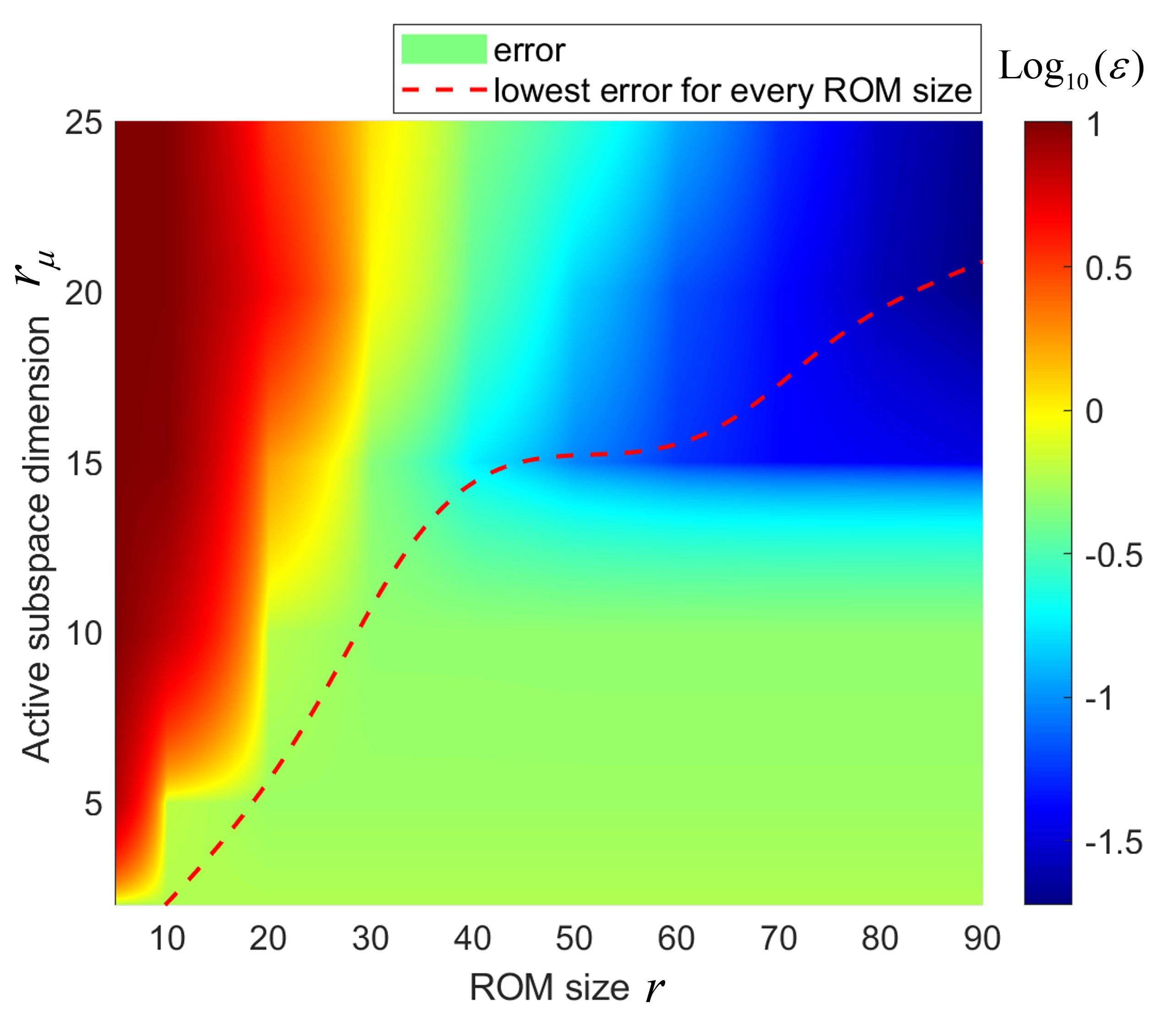}
    \caption{\centering MEMS actuator: relative error of AS ROM (Alg.~\ref{alg:ASpMOR}) changing with different combinations of active subspace dimension $r_\mu$ and ROM size $r$.}
    \label{fig:magact_25_UVError}
\end{figure}

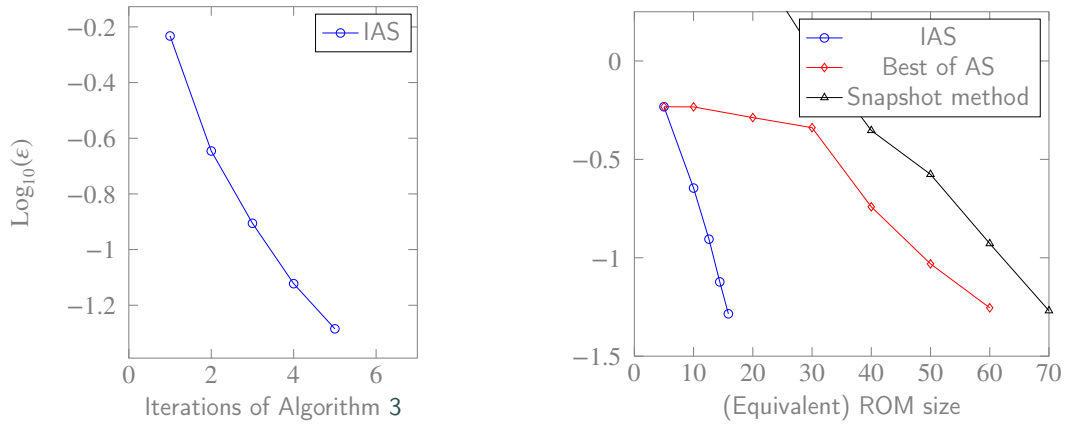
\begin{figure}[!htbp]
    \centering
    \hspace{-4em}\begin{subfigure}{0.5\textwidth}
    \centering
        \begin{tikzpicture}[scale=0.8]
            \begin{axis}[
                font=\large,
                width=2.5in,
                height=2.91in,
                xlabel={Iterations of Algorithm~\ref{alg:IASpMOR}},
                xmin=0,
                xmax=7,
                ylabel={$\operatorname{Log}_{10}(\varepsilon)$},
                ylabel style={yshift=3mm},
                title={},
                yticklabel style={xshift=-1mm}
            ]
            \addplot [color=blue, mark=o]
            coordinates {
            (1,   -0.232608224)
            (2,   -0.645833427)
            (3,   -0.905530704)
            (4,   -1.122551734)
    	  (5,   -1.284723842)
            };
            \legend{IAS}
            \end{axis}
        \end{tikzpicture}
    \end{subfigure}
    \hspace{-1em}\begin{subfigure}{0.5\textwidth}
    \centering
    \begin{tikzpicture}[scale=0.8]
        \begin{axis}[
            font=\large,
            xlabel={(Equivalent) ROM size},
            xmin=0,
            xmax=70,
            ymax=0.25,
            ymin=-1.5,
            ylabel style={yshift=3mm},
            title={},
            yticklabel style={xshift=-1mm}
        ]
        \addplot [color=blue, mark=o]
        coordinates {
            (5,         -0.232608224)
            (10,        -0.645833427)
            (12.5992,   -0.905530704)
            (14.4225,   -1.122551734)	
    	(15.8740,   -1.284723842)
        };

        \addplot [color=red, mark=diamond]
        coordinates {
          (5,    -0.232608224846068)
          (10,	-0.233492665018051)
          (20,	-0.287737616172197)
          (30,	-0.338852476866086)
          (40,	-0.741007093289210)
          (50,    -1.03095529408348)
          (60,	-1.25429399602457)
         };

        \addplot [color=black, mark=triangle]
        coordinates {
          (5,    1.00802629087934)
          (10,	0.999004217292526)
          (20,	0.522896255941952)
          (30,	0.0509886724370720)
          (40,	-0.353077602766038)
          (50,    -0.576108376536962)
          (60,	-0.927912273257516)
          (70,	-1.26914756067130)
         };
         
        \legend{IAS, Best of AS, Snapshot method}
        \end{axis}
    \end{tikzpicture}
    \end{subfigure}
    \caption{\centering MEMS actuator: relative error $\varepsilon$ decay of IAS ROM w.r.t. iterations (left) and the (equivalent) ROM sizes (right).}
    \label{fig:ias_decrease_25}
\end{figure}

\begin{small}
   \begin{table}[!htbp]
      \begin{center}
      \renewcommand\arraystretch{1.5}
        \caption{MEMS actuator: time comparison for solving FOM and ROMs with $\varepsilon\approx5\%$. }
        \begin{tabular}{cccc}
          \hline
          \textbf{{FOM}}& \makecell[c]{\textbf{AS-ROM}} & \makecell[c]{\textbf{IAS-ROMs}}& \makecell[c]{\textbf{IAS-ROMs}}\vspace{-1ex}\\
          (sparse solver) & & (in serial)& (in parallel)\\
          \hline
          2.23 s & 0.59 s & 0.069 s & 0.016 s
          \\
          \hline
        \end{tabular}
        \label{table:mag_static}
      \end{center}
    \end{table} 
\end{small}
\renewcommand{\arraystretch}{1}
\FloatBarrier

The standard AS method (Alg.~\ref{alg:ASpMOR}) requires 129 minutes offline for gradient computation in Step 3, excluding parameter searches for optimal \((r_\mu, r)\) combinations. Including the optimal \((r_\mu, r)\) search in Figure~\ref{fig:magact_25_UVError}, the total time increases significantly to 529 minutes. For the proposed IAS method (Alg.~\ref{alg:IASpMOR}), the first iteration includes gradient computation of the FOM (129 minutes, matching AS), while the subsequent 4 iterations have only additional gradient computations of the ROM (see also (\ref{eq:lazy_gradient})) according to Section 3.5, requiring only 12 minutes in total. This makes the offline time of IAS comparable to AS, while simultaneously improving accuracy and computational efficiency.  

\FloatBarrier
\subsection{MEMS Accelerometer}
Another common case for optimization and analysis is the MEMS accelerometer, which is widely used in navigation systems for automobiles, airplanes, and other vehicles. Through modeling a simplified MEMS accelerometer fabricated by our laboratory, we conducted a model with 188 parameters. To simplify the model, we excluded the electrostatic comb teeth in the structure, which are designed to generate capacitance changes through displacement. The simplified device model is shown in Figure~\ref{fig:XBW-acc}.

\begin{figure}[!htbp]
    \centering
    \includegraphics[width=0.9\linewidth]{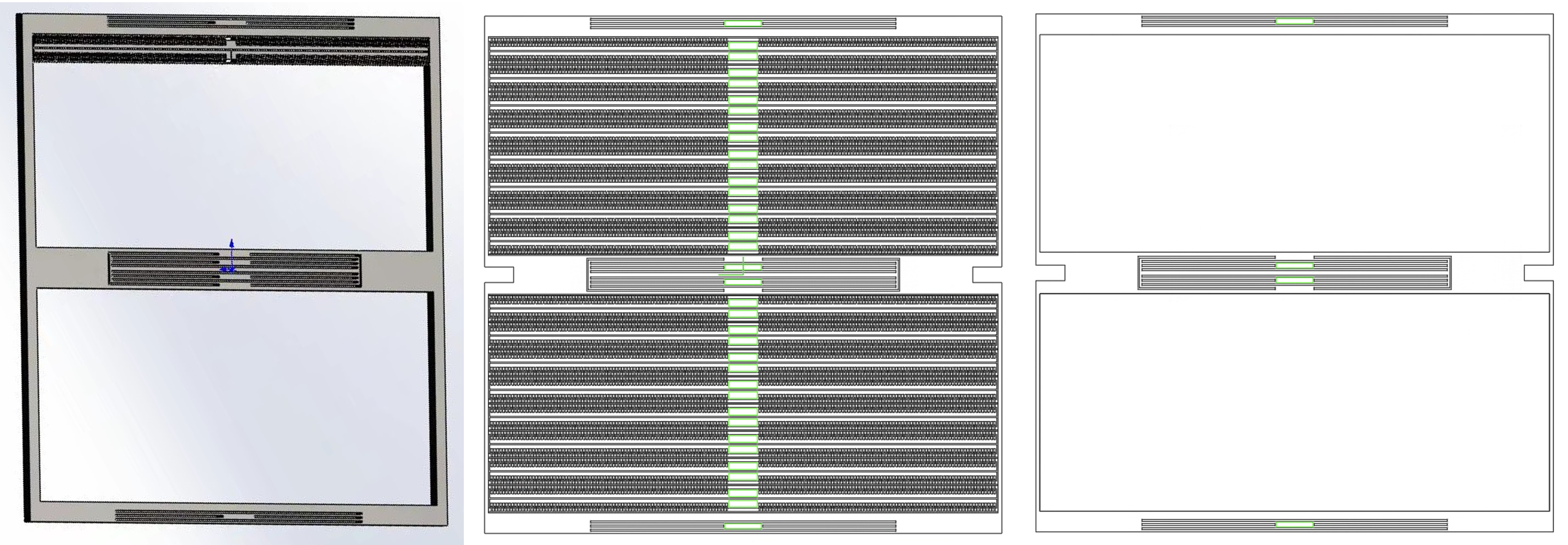}
    \caption{\centering 3D (left), 2D (middle) and 2D simplified (right) model of the accelerometer.}
    \label{fig:XBW-acc}
\end{figure}

For a simpler parametrization process, we introduced parameters in both the x and y directions for the 233 vertices of the boundary polygons, as the structure's shape is defined by these closed polygons. To avoid unreasonable design scenarios, we manually applied linear constraints to these parameters, reducing their total number to 188. We maintained the mesh points on the boundary of the device structure as straight segments. Inside the device structure, we parameterized the mesh points by proposing a spring-based smoothing method, explained in detail in Appendix~\ref{app1}. Figure~\ref{fig:7} illustrates the deformation of the parametric mesh by comparing two configurations: the original mesh (red), where all parameters are zero, and the deformed mesh (blue), where parameters $\mu_5$, $\mu_{16}$, $\mu_{77}$, and $\mu_{78}$ are set to 50~$\mu$m while all other parameters remain zero.

\begin{figure}[!htbp]
    \centering
    \includegraphics[width=0.36\linewidth]{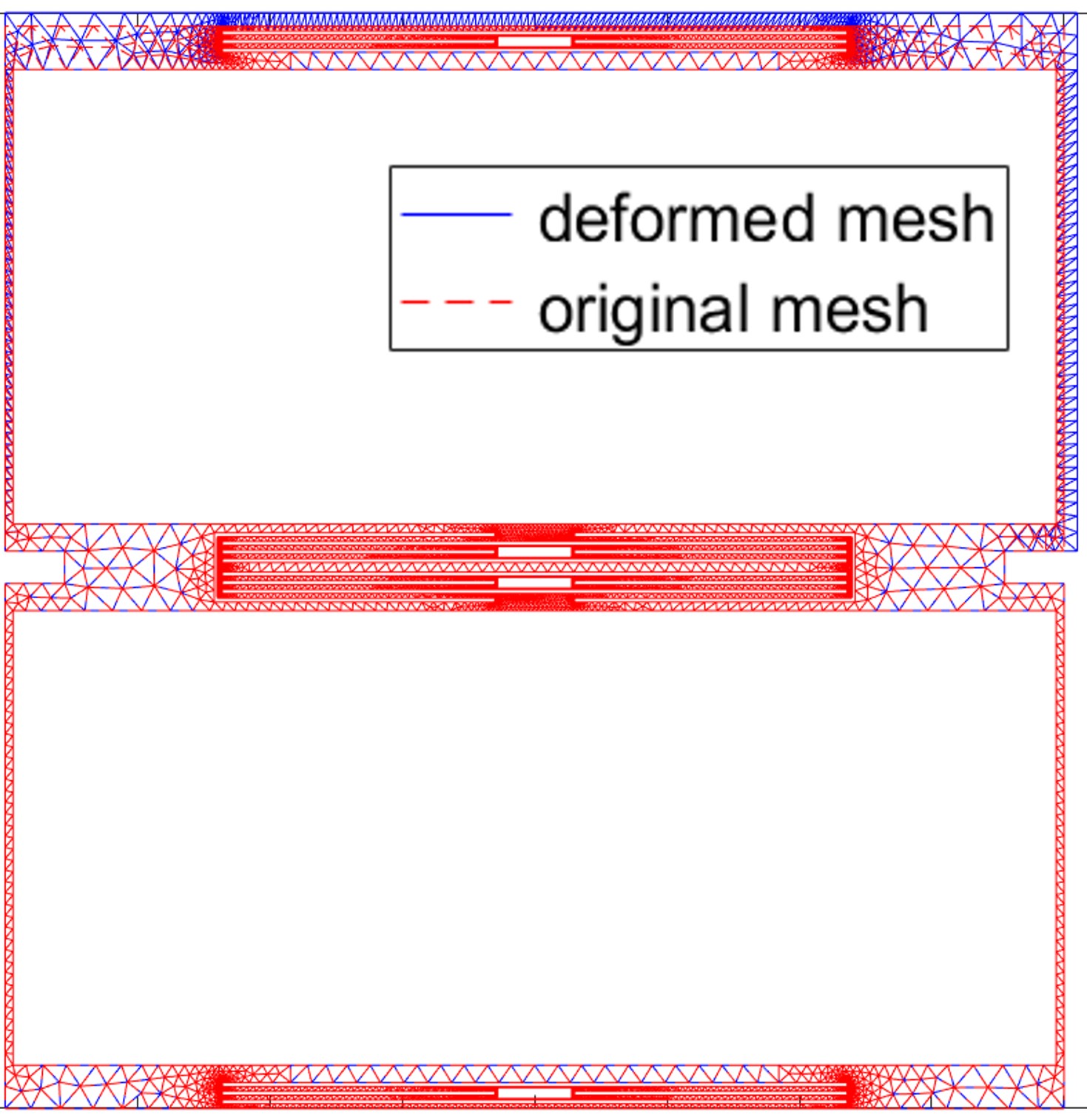}
    \caption{\centering{Parametric mesh of the accelerometer: original mesh (all parameters $=0$); deformed mesh ($\mu_{5},\mu_{16},\mu_{77},\mu_{78}=50$, others $=0$)}.}
    \label{fig:7}
\end{figure}

\begin{figure}[!htbp]
     \centering
     \includegraphics[width=0.36\linewidth]{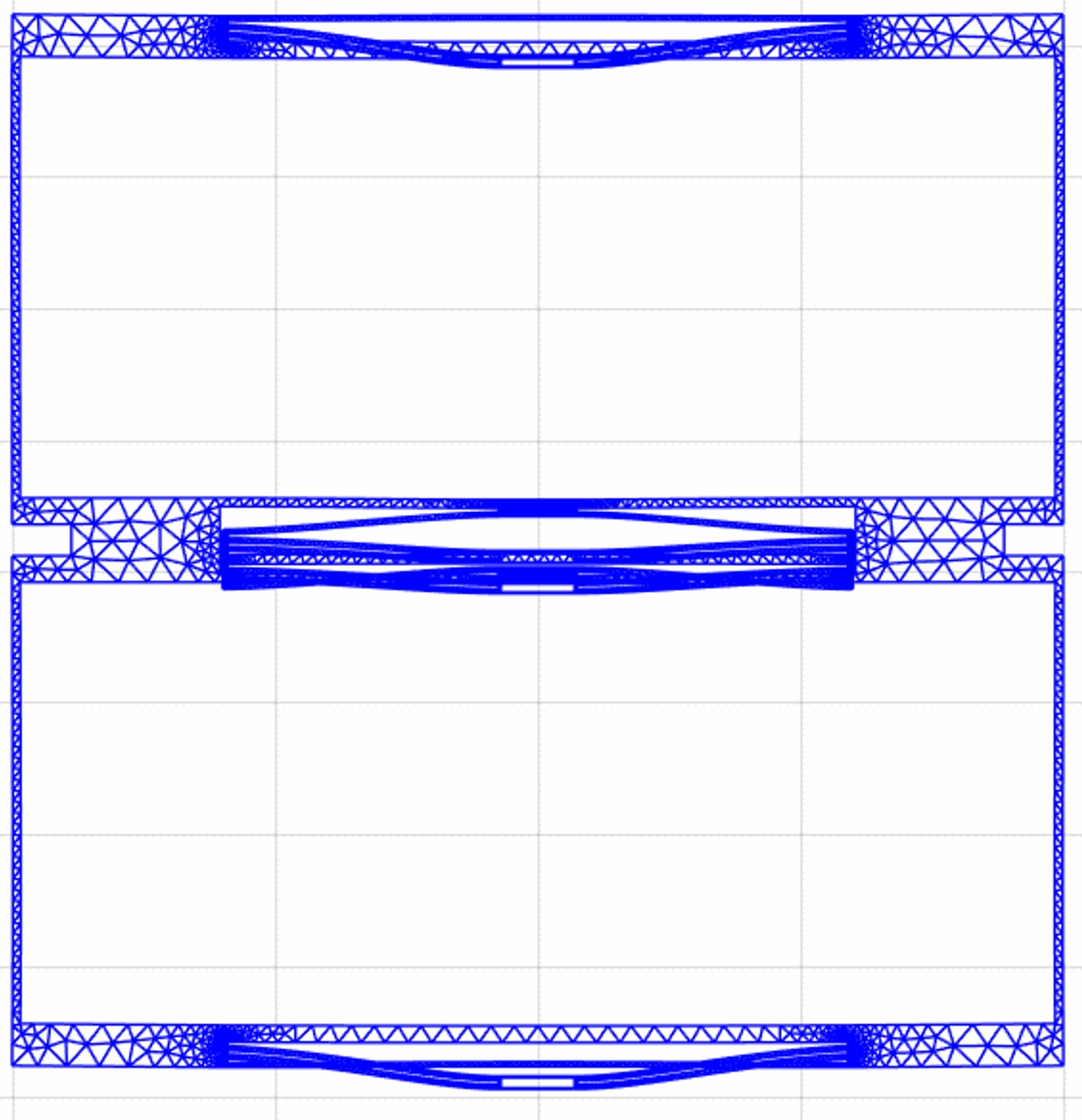}
     \caption{\centering Static deformation of the accelerometer under 10G acceleration in the vertical direction (Displacement with a magnification of $500\times$).}\label{Fig:Data2}
\end{figure}

Thus, we derive a second-order parametric system (\ref{eq:second_order_pLTI}) with $n=15,652$ degrees of freedom and \(n_\mu=188\). $C(\mu)=I$ is an identical matrix, \(B(\mu) \in  \mathbb R^{n}\) is the force on the structure when there is a 10G (gravitational) acceleration in the vertical direction, and \(x \in  \mathbb R^{n}\) is the displacement of all mesh points. Figure~\ref{Fig:Data2} shows the static deformation of the accelerometer under 10G acceleration in the vertical direction. The transfer function of this system is shown in (\ref{eq:trans_func}), where $Q(\mu,s)=s^2M(\mu)+sD(\mu)+K(\mu)$.

For an arbitrary sample point $(\mu^j,s^j)$, $\mu^j=[{\mu_1}^j,{\mu_2}^j,\ldots,{\mu_{n_\mu}}^j]^T$, the derivative of $H(\mu^j,s^j)$ can be written as:
\renewcommand{\arraystretch}{1.5}
    \begin{align}
        \hspace{-2ex}\begin{array}{ll}
             &\frac{\partial H(\mu^j,s^j)}{\partial {\mu_k}^j}=Q(\mu^j,s^j)^{-1}\frac{\partial B(\mu^j)}{\partial {\mu_k}^j}-Q(\mu^j,s^j)^{-1}\frac{\partial Q(\mu^j,s^j)}{\partial {\mu_k}^j}Q(\mu^j,s^j)^{-1}B(\mu^j)\in \mathbb R^{n},\\
             &\nabla_{\mu}H(\mu^j,s^j)=[\frac{\partial H(\mu^j,s^j)}{\partial {\mu_1}^j},\frac{\partial H(\mu^j,s^j)}{\partial {\mu_2}^j},\ldots,\frac{\partial H(\mu^j,s^j)}{\partial {\mu_{n_\mu}}^j}]^T \in \mathbb R^{n_\mu\times n}.
        \end{array}
    \label{eq:derivative_second_order}
\end{align}
\renewcommand{\arraystretch}{1}

Similarly,
\begin{align}
\begin{array}{ll}
    \nabla_{\mu}H_{r_i}(\mu^j,s^j)&=[\frac{\partial H_{r_i}(\mu^j,s^j)}{\partial {\mu_1}^j},\frac{\partial H_{r_i}(\mu^j,s^j)}{\partial {\mu_2}^j},\ldots,\frac{\partial H_{r_i}(\mu^j,s^j)}{\partial {\mu_{n_\mu}}^j}]^T \in \mathbb R^{n_\mu\times n},\\
    \nabla_{\mu}E_{i}(\mu^j,s^j)&=[\frac{\partial (H(\mu^j,s^j)-H_{r_i}(\mu^j,s^j))}{\partial {\mu_1}^j},\ldots,\frac{\partial H(\mu^j,s^j)-\partial H_{r_i}(\mu^j,s^j)}{\partial {\mu_{n_\mu}}^j}]^T \in \mathbb R^{n_\mu\times n}.
\end{array}
\end{align}

From the definition of $\nabla_{\mu}H(\mu^j,s^j)$ and $\nabla_{\mu}E_i(\mu^j,s^j)$, we sample $M=120$ parameter points for computing $\hat C$ in step 3 of Alg.~\ref{alg:ASpMOR} as $$\hat C=\frac{1}{M} \sum_{j=1}^{M} (\nabla_{\mu}H(\mu^j,s^j))^T(\nabla_{\mu}H(\mu^j,s^j)),$$ and computing $\hat C$ in step 3 of Alg.~\ref{alg:IASpMOR} as $$\hat C=\frac{1}{M} \sum_{j=1}^{M} (\nabla_{\mu}E_i(\mu^j,s^j))^T(\nabla_{\mu}E_i(\mu^j,s^j)).$$

We implement Alg.~\ref{alg:ASpMOR}, the standard AS method, and evaluate different combinations of active subspace dimension \(r_\mu\) and ROM size \(r\). For each fixed $r$, the approximation error presented in Figure~\ref{fig:Single active subspace 188} exhibits a similar characteristic pattern as observed in Figure~\ref{fig:magact_25_UVError}, and is characterized by two phases: before reaching the minimum value, the error decays with increasing $r_\mu$; afterwards, the error increases with further $r_\mu$ augmentation. Figure~\ref{fig:Single active subspace 188} visualizes this behavior through two key elements: (1) an error mesh with different \((r_\mu,r)\) combinations, and (2) a dashed line indicating the \(r_\mu\) value corresponding to the lowest error for each ROM size \(r\). When \(r_\mu = n_\mu\), the black line in Figure~\ref{fig:ias_decrease_188} (right) shows that the active subspace method is mathematically equivalent to the snapshot method, as no parameter space compression occurs. At smaller ROM sizes $r$, the active subspace method with reduced parameter dimensions (\(r_\mu < n_\mu\)) demonstrates a clear advantage, achieving reduced-order models with reasonably controlled errors. In contrast, the snapshot method (corresponding to the AS method with \(r_\mu = n_\mu\)) fails entirely to produce models with $\varepsilon<1$ with the same ROM size. This highlights the critical role of parameter-space dimensionality reduction for pMOR with large-dimensional parameter spaces. 
\begin{figure}[!htbp]
    \centering
    \includegraphics[width=0.75\linewidth]{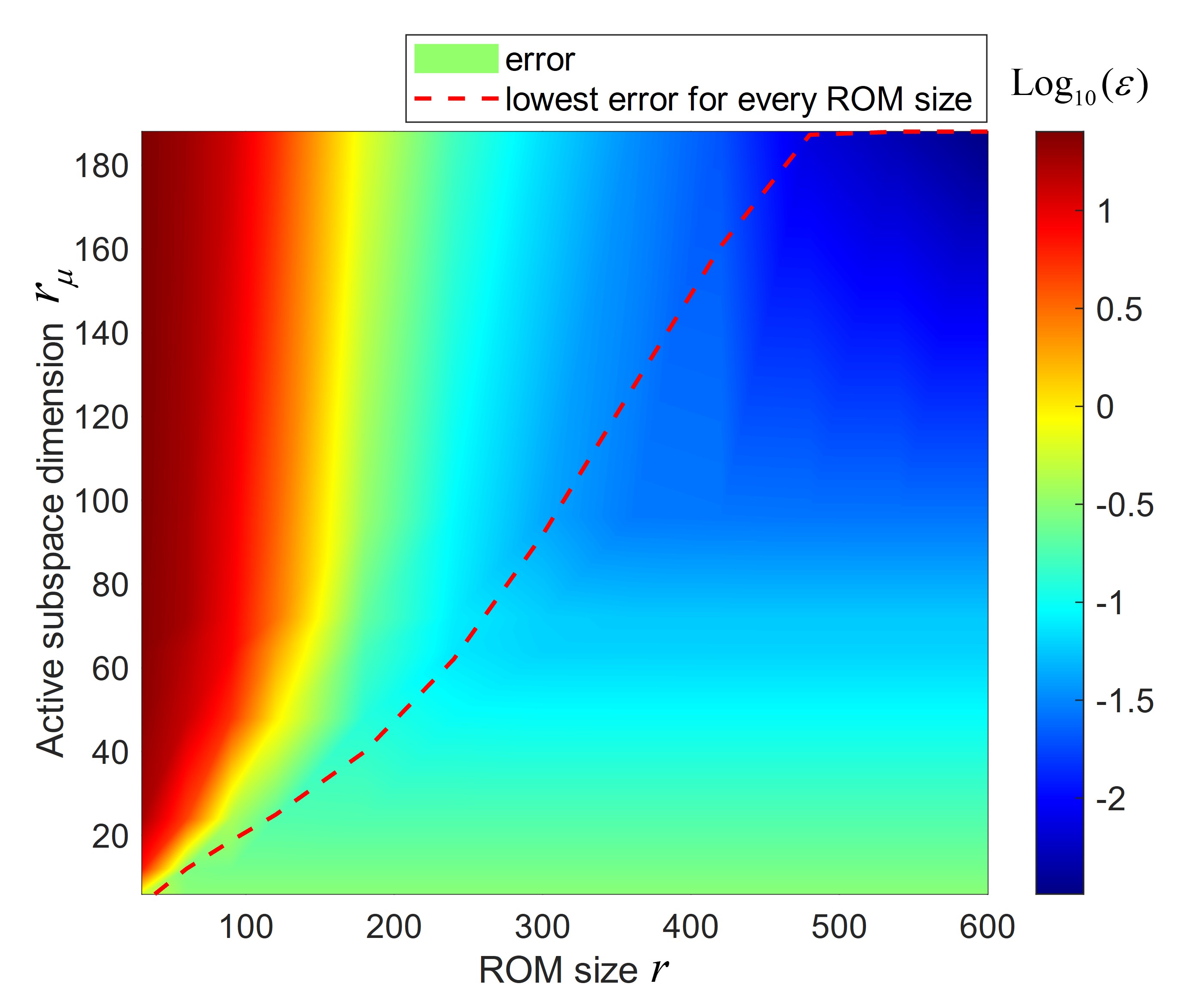}
    \vspace{-2ex}\caption{\centering MEMS accelerometer: relative error of AS ROM (Alg.~\ref{alg:ASpMOR}) changing with different combinations of active subspace dimension $r_\mu$ and ROM size $r$.}
    \label{fig:Single active subspace 188}
\end{figure}
\FloatBarrier
However, as the ROM size $r$ increases, the advantages of the active subspace method gradually diminish (Figure~\ref{fig:ias_decrease_188}). Notably, when higher accuracy requirements are imposed, the performance gain from active subspace becomes less significant. 

By employing the proposed iterative active subspace method (Alg.~\ref{alg:IASpMOR}), with $M_{V} = \frac{1}{2}r_{\mu,i}^2$ and $r_i = 3 \cdot r_{\mu,i}$ in the post-processing phase, we can achieve a much smaller error with the combination of a few smaller ROMs. The result is shown in Figure~\ref{fig:ias_decrease_188}. We can see that the relative error $\varepsilon$ decreases exponentially with the iterations. After 13 iterations with 13 ROMs, we achieve $\varepsilon\approx4.27\%$. Solving these ROMs one by one amounts to solving a single ROM with size $r\approx 59.7$. In contrast, with a single active subspace (Alg.~\ref{alg:ASpMOR}), we achieve similar accuracy ($\varepsilon\approx4.30\%$) with a ROM size of $r=300$. Moreover, the 13 ROMs derived from our iterative method can be solved in parallel, resulting in even greater computational gains. Table~\ref{table:188_rom_time} lists the runtimes of transient solutions to the FOM, the ROM from the standard AS method, and our IAS-ROM, at a testing parameter sample $\mu^*=[\mu_1, \ldots, \mu_{188}]^T$ in a time period of $[0, 500]$ microseconds.

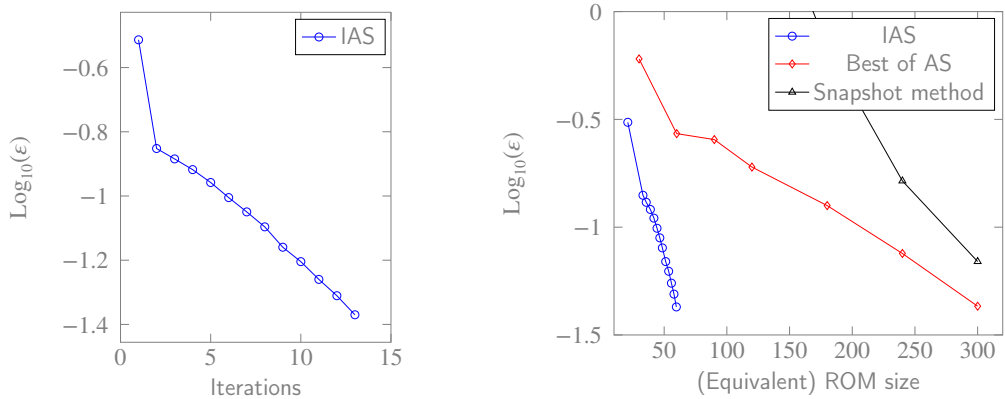
\begin{figure}[!htbp]
    \centering
    \hspace{-3em}\begin{subfigure}{0.45\textwidth}
    \centering
        \begin{tikzpicture}[scale=0.75]
            \begin{axis}[
                font=\large,
                width=2.5in,
                height=2.91in,
                xlabel={Iterations},
                xmin=0,
                xmax=15,
                ylabel={$\operatorname{Log}_{10}(\varepsilon)$},
                ylabel style={yshift=3mm},
                title={},
                yticklabel style={xshift=-1mm}
            ]
            \addplot [color=blue, mark=o]
            coordinates {
            (1, -0.513528641649309)
            (2, -0.852051224598756)
            (3, -0.884516749989795)
            (4, -0.917874049674271)
            (5, -0.957998589263179)
            (6, -1.00499499432448)	
            (7, -1.04953419033504)
            (8, -1.09612392656925)
            (9, -1.15953030896753)
            (10, -1.20451208096091)
            (11, -1.25975312072839)
            (12, -1.31059469291642)
            (13, -1.36999165531536)
            
            };
            \legend{IAS}
            \end{axis}
        \end{tikzpicture}
    \end{subfigure}
    \hspace{-1em}\begin{subfigure}{0.45\textwidth}
    \centering
    \begin{tikzpicture}[scale=0.75]
        \begin{axis}[
            font=\large,
            xlabel={(Equivalent) ROM size},
            xmin=10,
            xmax=320,
            ymax=0,
            ymin=-1.5,
            ylabel={$\operatorname{Log}_{10}(\varepsilon)$},
            ylabel style={yshift=3mm},
            title={},
            yticklabel style={xshift=-1mm}
        ]

        \addplot [color=blue, mark=o]
        coordinates {							

        (21.0000000000000, -0.513528641649309)
        (33.0988769953938, -0.852051224598756)
        (35.7059378584483, -0.884516749989795)
        (39.0059162622267, -0.917874049674271)
        (41.8258091718356, -0.957998589263179)
        (44.3094575618814, -1.00499499432448)	
        (46.5421054626723, -1.04953419033504)
        (48.5789270018360, -1.09612392656925)
        (51.2136367078195, -1.15953030896753)
        (53.6020122508827, -1.20451208096091)
        (55.7946347020751, -1.25975312072839)
        (57.8272704894256, -1.31059469291642)
        (59.7262529419056, -1.36999165531536)

        };

        \addplot [color=red, mark=diamond]
        coordinates {
        						
        (30,   -0.219478445493176)
        (60,  -0.566038405214920)
        (90,  -0.593661540041240)
        (120,  -0.721119477303452)
        (180,  -0.899882259131360)
        (240,  -1.12198280060338)
        (300,  -1.36667622112262)
        };

        \addplot [color=black, mark=triangle]
        coordinates{

        (30, 1.40247965625539)	
        (60, 1.28056798103654)	
        (90, 1.09052528722691)	
        (120, 0.734472017464898)	
        (180, -0.168571127196827)	
        (240, -0.785392228202909)	
        (300, -1.15939485304881)
        };
        \legend{IAS, Best of AS, Snapshot method}
        \end{axis}
    \end{tikzpicture}
    \end{subfigure}
    \caption{\centering MEMS accelerometer: relative error $\varepsilon$ decay with iterations (left) and (equivalent) ROM sizes (right).}
    \label{fig:ias_decrease_188}
\end{figure}
\begin{small}
   \begin{table}[!htbp]
      \begin{center}
      \renewcommand\arraystretch{1.5}
        \caption{\centering MEMS accelerometer: time comparison for solving FOM and ROMs with $\varepsilon\approx4.3\%$. }
        \begin{tabular}{cccc}
          \hline
          \textbf{{FOM}}& \makecell[c]{\textbf{AS-ROM}} & \makecell[c]{\textbf{IAS-ROMs}}& \makecell[c]{\textbf{IAS-ROMs}}\vspace{-1ex}\\
          (sparse solver) & & (in serial)& (in parallel)\\
          \hline
          1.45 s & 0.328 s & 0.0981 s & 0.00795 s
          \\
          \hline
        \end{tabular}
        \label{table:188_rom_time}
      \end{center}
    \end{table} 
\end{small}
\renewcommand{\arraystretch}{1}		
\FloatBarrier

Considering the offline time, the standard AS method (Alg.~\ref{alg:ASpMOR}) requires about 108 minutes (excluding parameter search for an optimal $(r_\mu, r)$ combination). The $(r_\mu, r)$ search in Figure~\ref{fig:Single active subspace 188} leads to a total time of 313 minutes. According to Section 3.5, for the proposed IAS method (Alg.~\ref{alg:IASpMOR}), the 13 iterations shown in Figure~\ref{fig:ias_decrease_188} involve gradient computation of the FOM during the first iteration (108 minutes, same as Alg.~\ref{alg:ASpMOR}) and gradient computation of the ROM at subsequent iterations (372 minutes for 12 iterations). While IAS exhibits a moderate increase in the offline time compared to AS, it achieves enhanced accuracy and greater computational gain in the online phase.

\FloatBarrier
\section{Conclusion}

In this paper, we propose an iterative active subspace method for projection-based parametric model order reduction. This approach outperforms the existing active subspace method by iteratively generating multiple small active subspace ROMs. The final ROM is the sum of the multiple small ROMs derived at all previous iterations. Consequently, simulating the final ROM can be done via simulating the multiple ROMs in parallel, which is much more efficient than simulating the single big ROM obtained from the standard active subspace method. The experimental results show that our proposed method is a robust tool for parametric model order reduction of problems with high-dimensional parameter spaces.

\section{Data availability}
 The code and data will be made available on Zenodo upon publication.

\appendix
\numberwithin{equation}{section}
\numberwithin{figure}{section}
\numberwithin{table}{section}

\section{Proofs of Theorems~\ref{theo} and~\ref{theo_all}}
\label{app:proofs}

\subsection{Proof of Theorem~\ref{theo}}
\label{app:proof_theo}
\begin{proof}[Proof of Theorem~\ref{theo}]

Assume \(U_i=U_{i-1}\), s.t. \(V_i=V_{i-1}\). Since \(U_\tau,\tau=1,2,\ldots,\) are orthogonal matrices, we have:
\[U_iU_i^TU_{i-1}U_{i-1}^T=U_iU_i^T=U_{i-1}U_{i-1}^T.\]
From (\ref{eq:IASpMOR_general}) and (\ref{eq:IASpMOR_general_E}),
\begin{align}
     \begin{array}{llll}
     E_{r_{i-1}}(\mu,s)&=\mathrm{ASpMOR}_{V_i}^{U_i}(H(\mu,s))-H_{r_{i-1}}(U_iU_i^T\mu,s)
     \\&=\mathrm{ASpMOR}_{V_i}^{U_i}(H(\mu,s))
     -H_{r_{i-2}}(U_iU_i^T\mu,s)\\& \hspace{2.4ex}-\hspace{0.5ex} \mathrm{ASpMOR}_{V_{i-1}}^{U_{i-1}}(H(U_iU_i^T\mu,s))\\& \hspace{2.4ex}+ \hspace{0.5ex} \mathrm{ASpMOR}_{I_{\text{\hspace{1.05ex} }}}^{U_{i-1}}(H_{r_{i-2}}(U_iU_i^T\mu,s)).
    \end{array}
    \label{eq:E_r_i-1_full}
\end{align}
$U_i=U_{i-1}$ and $V_i=V_{i-1}$ lead to:
    \[\mathrm{ASpMOR}_{V_{i-1}}^{U_{i-1}}(H(U_iU_i^T\mu,s))=\mathrm{ASpMOR}_{I_{\text{\hspace{1.05ex} }}V_{i-1}}^{U_iU_{i-1}}(H(\mu,s))=\mathrm{ASpMOR}_{V_{i}}^{U_{i}}(H(\mu,s)),\] 
and
    \[\mathrm{ASpMOR}_{I_{\text{\hspace{1.05ex} }}}^{U_{i-1}}(H_{r_{i-2}}(U_iU_i^T\mu,s))=H_{r_{i-2}}(U_{i-1}U_{i-1}^TU_iU_i^T\mu,s)=H_{r_{i-2}}(U_iU_i^T\mu,s).\]
\hspace{-1ex}Thus, the last two terms in (\ref{eq:E_r_i-1_full}) cancel out the first two terms, respectively, making \(E_{r_{i-1}}(\mu,s)=0\).
\end{proof}

\subsection{Proof of Theorem~\ref{theo_all}}
\label{app:proof_theo_all}
\begin{proof}[Proof of Theorem~\ref{theo_all}]
If \({U_l}^TU_j=0\), \(l\neq j\), \(l,j\leq i-1\), then following (\ref{eq:IASpMOR_general}) and (\ref{eq:IASpMOR_general_E}),
\renewcommand{\arraystretch}{1.5} 
\begin{align}
    \begin{array}{rl}
         H_{r_{i-1}}(\mu,s)&=H_{r_{i-2}}(\mu,s)+\mathrm{ASpMOR}_{V_{i-1}}^{U_{i-1}}(H(\mu,s))-\mathrm{ASpMOR}_{I_{\text{\ }}}^{U_{i-1}}(H_{r_{i-2}}(\mu,s))\\&=H_{r_{i-2}}(\mu,s)+\mathrm{ASpMOR}_{V_{i-1}}^{U_{i-1}}(H(\mu,s))-\mathrm{ASpMOR}_{I_{\text{\ }}}^{U_{i-1}}(H_{r_{i-3}}(\mu,s))\\&\hspace{2.5ex}-\mathrm{ASpMOR}_{I_{\text{\hspace{3ex} }}V_{i-2}}^{U_{i-1}U_{i-2}}(H(\mu,s))+\mathrm{ASpMOR}_{I_{\text{\hspace{3ex} }}I_{\text{\ }}}^{U_{i-1}U_{i-2}}(H_{r_{i-3}}(\mu,s))\\
         &=H_{r_{i-2}}(\mu,s)+\mathrm{ASpMOR}_{V_{i-1}}^{U_{i-1}}(H(\mu,s))-\mathrm{ASpMOR}_{I_{\text{\ }}}^{U_{i-1}}(H_{r_{i-3}}(\mu,s))\\&\hspace{2.5ex}-\mathrm{ASpMOR}_{V_{i-2}}^{I_{\text{\ }}}(H(0,s))+H_{r_{i-3}}(0,s).
    \end{array}
    \label{eq:IASpMOR_general_orth_half}
\end{align}
\renewcommand{\arraystretch}{1}

For an arbitrary $j\leq i-1$, substituting $\mu=0$ in (\ref{eq:IASpMOR_general}) and (\ref{eq:IASpMOR_general_E}) gives
\renewcommand{\arraystretch}{1.5}\begin{align}
\begin{array}{rl}
     H_{r_{j}}(0,s)&=H_{r_{j-1}}(0,s)+\mathrm{ASpMOR}_{V_{j}}^{U_{j}}(H(0,s))-\mathrm{ASpMOR}_{I_{\text{\ }}}^{U_{j}}(H_{r_{j-1}}(0,s))  \\
     &=H_{r_{j-1}}(0,s)+\mathrm{ASpMOR}_{V_{j}}^{I_{\text{\ }}}(H(0,s))-H_{r_{j-1}}(0,s)\\
     &=\mathrm{ASpMOR}_{V_{j}}^{I_{\text{\ }}}(H(0,s)).
\end{array}
\label{eq:H_r_zero_simp}
\end{align}
\renewcommand{\arraystretch}{1}

In the last equality of (\ref{eq:IASpMOR_general_orth_half}), replacing $H_{r_{i-3}}(\mu,s)$ with its expression from (\ref{eq:IASpMOR_general}), and applying (\ref{eq:H_r_zero_simp}) to $H_{r_{i-3}}(0,s)$ leads to
\renewcommand{\arraystretch}{1.5} 
\begin{align}
    \begin{array}{rl}
         H_{r_{i-1}}(\mu,s)
         &=H_{r_{i-2}}(\mu,s)+\mathrm{ASpMOR}_{V_{i-1}}^{U_{i-1}}(H(\mu,s))-\mathrm{ASpMOR}_{I_{\text{\ }}}^{U_{i-1}}(H_{r_{i-3}}(\mu,s))\\&\hspace{2.5ex}-\mathrm{ASpMOR}_{V_{i-2}}^{I_{\text{\ }}}(H(0,s))+\mathrm{ASpMOR}_{V_{i-3}}^{I_{\text{\ }}}(H(0,s))\\
         &=H_{r_{i-2}}(\mu,s)+\mathrm{ASpMOR}_{V_{i-1}}^{U_{i-1}}(H(\mu,s))-\mathrm{ASpMOR}_{I_{\text{\ }}}^{U_{i-1}}(H_{r_{i-4}}(\mu,s))\\&\hspace{2.5ex}-\mathrm{ASpMOR}_{I_{\text{\hspace{3ex} }}V_{i-3}}^{U_{i-1}U_{i-3}}(H(\mu,s))+\mathrm{ASpMOR}_{I_{\text{\hspace{3ex} }}I_{}}^{U_{i-1}U_{i-3}}(H_{r_{i-4}}(\mu,s))\\&\hspace{2.5ex}-\mathrm{ASpMOR}_{V_{i-2}}^{I_{\text{\ }}}(H(0,s))+\mathrm{ASpMOR}_{V_{i-3}}^{I_{\text{\ }}}(H(0,s))\\
         &=H_{r_{i-2}}(\mu,s)+\mathrm{ASpMOR}_{V_{i-1}}^{U_{i-1}}(H(\mu,s))-\mathrm{ASpMOR}_{I_{\text{\ }}}^{U_{i-1}}(H_{r_{i-4}}(\mu,s))\\&\hspace{2.5ex}-\mathrm{ASpMOR}_{V_{i-3}}^{I}(H(0,s))+H_{r_{i-4}}(0,s)\\&\hspace{2.5ex}-\mathrm{ASpMOR}_{V_{i-2}}^{I_{\text{\ }}}(H(0,s))+\mathrm{ASpMOR}_{V_{i-3}}^{I_{\text{\ }}}(H(0,s)).\\
         \end{array}
\end{align}
Repeatedly applying (\ref{eq:IASpMOR_general}) to all $H_{r_j}(\mu, s)$ and replacing $H_{r_j}(0, s)$ with $\mathrm{ASpMOR}^I_{V_j} H(0, s)$ for all $j\leq i-4$, we get

\begin{align}
        \begin{array}{ll}
         H_{r_{i-1}}(\mu,s)
         &=H_{r_{i-2}}(\mu,s)+\mathrm{ASpMOR}_{V_{i-1}}^{U_{i-1}}(H(\mu,s))-\mathrm{ASpMOR}_{I_{\text{\ }}}^{U_{i-1}}(H_{r_{i-4}}(\mu,s))\\&\hspace{2.5ex}-\mathrm{ASpMOR}_{V_{i-3}}^{I}(H(0,s))+\mathrm{ASpMOR}_{V_{i-4}}^{I_{\text{\ }}}(H(0,s))\\&\hspace{2.5ex}-\mathrm{ASpMOR}_{V_{i-2}}^{I_{\text{\ }}}(H(0,s))+\mathrm{ASpMOR}_{V_{i-3}}^{I_{\text{\ }}}(H(0,s))\\
         &\vdots\\
         &=H_{r_{i-2}}(\mu,s)+\mathrm{ASpMOR}_{V_{i-1}}^{U_{i-1}}(H(\mu,s))-\mathrm{ASpMOR}_{I_{\text{\ }}}^{U_{i-1}}(H_{r_{0}}(\mu,s))\\&\hspace{2.5ex}-\sum\limits_{j=1}\limits^{i-2}\mathrm{ASpMOR}_{V_{j}}^{I}(H(0,s))+\sum\limits_{j=0}\limits^{i-3}\mathrm{ASpMOR}_{V_{j}}^{I_{\text{\ }}}(H(0,s))\\
         &=H_{r_{i-2}}(\mu,s)+\mathrm{ASpMOR}_{V_{i-1}}^{U_{i-1}}(H(\mu,s)))-H_r(0,s)\\&\hspace{2.5ex}-\mathrm{ASpMOR}_{V_{i-2}}^{I}(H(0,s))+\mathrm{ASpMOR}_{V_{0}}^{I_{\text{\ }}}(H(0,s))~~~~~~~~~\text{\footnotesize{$(H_{r_0}(\mu,s)=H_r(0,s))$}}\\
         &=H_{r_{i-2}}(\mu,s)+\mathrm{ASpMOR}_{V_{i-1}}^{U_{i-1}}(H(\mu,s))-H_r(0,s)\\&\hspace{2.5ex}-\mathrm{ASpMOR}_{V_{i-2}}^{I}(H(0,s))+H_r(0,s)~~~~~~~~~~~~\text{\footnotesize{$(\mathrm{ASpMOR}_{V_0}^I(H(0,s))=H_r(0,s))$}}\\
         &=H_{r_{i-2}}(\mu,s)+\mathrm{ASpMOR}_{V_{i-1}}^{U_{i-1}}(H(\mu,s))-H_{r_{i-2}}(0,s).
    \end{array}
    \label{eq:IASpMOR_general_orth_half_further}
\end{align}

\renewcommand{\arraystretch}{1}

Equation (\ref{eq:IASpMOR_general_orth_half_further}) provides a recursive expression for \( H_{r_{i-1}}(\mu,s) \), from which we can derive the following accumulative expression:
\begin{align}
\hspace{-3ex}\begin{array}{rl}
     H_{r_{i-1}}(\mu,s)&=H_{r_{i-2}}(\mu,s)+\mathrm{ASpMOR}_{V_{i-1}}^{U_{i-1}}(H(\mu,s))-H_{r_{i-2}}(0,s)  \\
     & =H_{r_{i-3}}(\mu,s)+\mathrm{ASpMOR}_{V_{i-2}}^{U_{i-2}}(H(\mu,s))-H_{r_{i-3}}(0,s)\\&\hspace{13.8ex}+\hspace{0.5ex}\mathrm{ASpMOR}_{V_{i-1}}^{U_{i-1}}(H(\mu,s))-H_{r_{i-2}}(0,s)\\
     &\hspace{0.675ex}\vdots\\
     &=H_{r_0}(\mu,s)+\sum\limits_{j=1}\limits^{i-1}\mathrm{ASpMOR}_{V_{j}}^{U_{j}}(H(\mu,s))-\sum\limits_{j=0}\limits^{i-2}H_{r_j}(0,s)\hspace{1ex}~\text{\footnotesize{$(H_{r_0}(\mu,s)=H_{r_0}(0,s))$}}\\
     &=\sum\limits_{j=1}\limits^{i-1}\mathrm{ASpMOR}_{V_{j}}^{U_{j}}(H(\mu,s))-\sum\limits_{j=1}\limits^{i-2}\mathrm{ASpMOR}_{V_{j}}^{I}H(0,s),\hspace{1ex}i\geq2.
\end{array}
    \label{eq:accumulative_IAS_pMOR_half}
\end{align}

If $\exists~g<i-1$, $U_i=U_g$, s.t. $V_i=V_g$, then according to (\ref{eq:IASpMOR_general}) and (\ref{eq:accumulative_IAS_pMOR_half}),
\renewcommand{\arraystretch}{1.5} 
\begin{align}
    \hspace{-2ex}\begin{array}{rl}
         H_{r_{i}}(\mu,s)
         &=H_{r_{i-1}}(\mu,s)+\mathrm{ASpMOR}_{V_i}^{U_i}(H(\mu,s))-\mathrm{ASpMOR}_{I_{\text{\ }}}^{U_i}(H_{r_{i-1}}(\mu,s))\\
         &=H_{r_{i-1}}(\mu,s)+\mathrm{ASpMOR}_{V_i}^{U_i}(H(\mu,s))\\
         &\hspace{2.5ex}-\sum\limits_{j=1}\limits^{i-1}\mathrm{ASpMOR}_{I_{\hspace{1.5ex}}V_{j}}^{U_{i}U_{j}}(H(\mu,s))+\sum\limits_{j=1}\limits^{i-2}\mathrm{ASpMOR}_{V_{j}}^{U_i}H(0,s)\\
         &=H_{r_{i-1}}(\mu,s)+\mathrm{ASpMOR}_{V_i}^{U_i}(H(\mu,s))-\mathrm{ASpMOR}_{I_{\hspace{1.5ex}}V_{g}}^{U_{i}U_{g}}(H(\mu,s))\\
         &\hspace{2.5ex}-\sum\limits_{j=1,j\neq g}\limits^{i-1}\mathrm{ASpMOR}_{V_{j}}^{I}(H(0,s))+\sum\limits_{j=1}\limits^{i-2}\mathrm{ASpMOR}_{V_{j}}^{I}H(0,s)\\
         &= H_{r_{i-1}}(\mu,s)+\mathrm{ASpMOR}_{V_i}^{U_i}(H(\mu,s))-\mathrm{ASpMOR}_{V_{i}}^{U_{i}}(H(\mu,s))~~~\textbf{\footnotesize{$(U_g=U_i,~V_g=V_i)$}}\\
         &\hspace{2.5ex}-\mathrm{ASpMOR}_{V_{i-1}}^{I}(H(0,s))+\mathrm{ASpMOR}_{V_{g}}^{I}(H(0,s))\\
         &=H_{r_{i-1}}(\mu,s)-\mathrm{ASpMOR}_{V_{i-1}}^{I}(H(0,s))+\mathrm{ASpMOR}_{V_{i}}^{I}(H(0,s))~~~~~~~~~~~~~\textbf{\footnotesize{$(V_g=V_i)$}}\\
         &=H_{r_{i-1}}(\mu,s)-H_{r_{i-1}}(0,s)+H_{r_i}(0,s).
    \end{array}
\label{eq:IASpMOR_general_orth_simp}
\end{align}
\renewcommand{\arraystretch}{1}
\end{proof}

\section{Finite element modeling with parametric mesh}
\label{app1}
The implementation of parametric model order reduction (pMOR) necessitates a parametric representation of the underlying system. For finite element models, parametric forms are conventionally derived via data-driven approaches such as interpolation~\cite{morLieRK06,morBauBBetal11,morGeuPL13}, operator inference~\cite{morPehW16,morKraPW24}, etc. However, in scenarios involving high-dimensional structural parameters (e.g., geometric, material, or boundary condition parameters), these methods become infeasible due to the curse of dimensionality---exponential growth in required training data and computational costs as the number of parameters increases. 

To circumvent these limitations, this work adopts an analytical methodology based on a parametric mesh to explicitly construct the parametric model. By establishing a direct geometric mapping between structural parameters and nodal coordinates through the parametric mesh, we derive analytical expressions for parameter-dependent system matrices (e.g., stiffness matrices) without relying on data-driven approximations. Moreover, this approach ensures analytical differentiability of the system equations with respect to all geometric parameters~\cite{morFroGGetal19}, which is required by the active subspace method, as the parametric mesh constructed in this work guarantees continuously differentiable nodal coordinates with respect to design parameters. Our parametric FEM tool is implemented in MATLAB, and its symbolic computation function is powered by the self-contained symbolic framework of CasADi~\cite{AndGHetal19}.

During the design process, users can define deformation parameters $\mu=[\mu_1, \mu_2, \ldots, \mu_{n_\mu}]$ on the coordinates of the boundary points of the geometric model. The internal mesh points then move in accordance with the deformations of boundary points. A spring smoothing mechanism is employed to adjust the mesh model dynamically, avoiding topology changes during data generation.\vspace{-2ex}

\vspace{-2ex}\begin{align}
P(\mu)=\tilde{K}_s^{-1}{F_B}(\mu) = \tilde{K}_s^{-1} \cdot \left[\begin{array}{c}
0\\
\smash{\vdots}\\
0\\
{f_1(\mu)}\\
0\\
\smash{\vdots}\\
0\\
{f_2(\mu)}\\
0\\
\smash{\vdots}\\
\smash{\vdots}\\
\smash{\vdots}\\
0\\
{f_m(\mu)}\\
0\\
\smash{\vdots}
\end{array} \right] = \tilde{K}_s^{-1} \cdot \underbrace{\left[ \begin{array}{cccc}
0&0&&0\\
\smash{\vdots}&\smash{\vdots}&&\smash{\vdots}\\
0&\smash{\vdots}&&\smash{\vdots}\\
1&\smash{\vdots}&&\smash{\vdots}\\
0&\smash{\vdots}&&\smash{\vdots}\\
\smash{\vdots}&\smash{\vdots}&&\smash{\vdots}\\
\smash{\vdots}&0&&\smash{\vdots}\\
\smash{\vdots}&1&&\smash{\vdots}\\
\smash{\vdots}&0&&\smash{\vdots}\\
\smash{\vdots}&\smash{\vdots}&\ldots&\smash{\vdots}\\
\smash{\vdots}&\smash{\vdots}&&\smash{\vdots}\\
\smash{\vdots}&\smash{\vdots}&&\smash{\vdots}\\
\smash{\vdots}&\smash{\vdots}&&0\\
\smash{\vdots}&\smash{\vdots}&&1\\
\smash{\vdots}&\smash{\vdots}&&0\\
\smash{\vdots}&\smash{\vdots}&&\smash{\vdots}
\end{array} \right]}_{\mathlarger{\tilde{F}}} \cdot \underbrace{\left[\begin{array}{cccc}
{f_1(\mu)}\\
{f_2(\mu)}\\
\vdots\\
{f_m(\mu)}
\end{array}\right]}_{\mathlarger{\tilde{f}(\mu)}}.
\label{eq_para_mesh}
\end{align}

The construction of the overall deformation spring stiffness matrix, $K_s$, is facilitated by dividing the domain into triangular meshes and attaching a spring to each edge of the mesh. The boundary deformation parameters are set as fixed boundary conditions $ F_B(\mu) = [0, \ldots, f_1(\mu), \ldots, 0, \ldots,  f_2(\mu), \ldots]^T$, where $f^j(\mu),j=1,\ldots,m,$ are the functions describing how the boundary points move with the parameters $\mu$. By imposing fixed boundary displacement conditions, \( K_s \) is processed (rows of \( K_s \) with displacement parameters are replaced with corresponding rows from the identity matrix \( I \)), resulting in the modified stiffness matrix \( \tilde{K}_s \). Solving the parametric linear system \( \tilde{K}_s P(\mu) = F_B(\mu) \) yields explicit parametric expressions for each internal mesh point (see~\ref{eq_para_mesh}). Given the sparsity of $F_B(\mu)$, we only need to solve a linear system with $m$ right-hand sides to obtain $\tilde P:=\tilde K_s^{-1} \tilde F$ once, solution $P(\mu)$ at any sample $\mu^\ast$ of the parameter $\mu$ can then be computed by a single matrix multiplication $P(\mu^\ast)=\tilde P \tilde f(\mu^\ast)$.

This process provides explicit parametric expressions $P(\mu)$ for each mesh point. We call the proposed parametric modeling approach the spring smoothing method. Figure~\ref{fig:Deform_L_mesh} shows an example of an L-shaped structure whose corner points of the contour are manually parametric-defined, and when the parameter changes, the other mesh points move with the corner points according to the spring smoothing method. 

Finally, the elemental stiffness matrix $k_e(\mu)_i$ can be constructed with $P(\mu)$ and the invariant topological relationships as below
\begin{align}
    k_e(\mu)_i = {B_i(\mu)^T}DB_i(\mu)tA_i(\mu).
    \label{equ:element_k}
\end{align}
Here $B_i(\mu)$ is a matrix that contains shape function derivatives for the $i$-th triangular element of the mesh, $D$ is the elasticity matrix, $t$ and $A_i(\mu)$ are the thickness and area of that element.

Thus, the stiffness matrix \( K(\mu) \) and the mass matrix \( M(\mu) \) can be established by assembling $k_e(\mu)_i$. While  $B_i(\mu)$ and $A_i(\mu)$ are differentiable w.r.t. the parameters $\mu$, \( K(\mu) \) and \( M(\mu) \) are differentiable.

\begin{figure}[!htbp]
    \centering
    \includegraphics[width=0.8\linewidth]{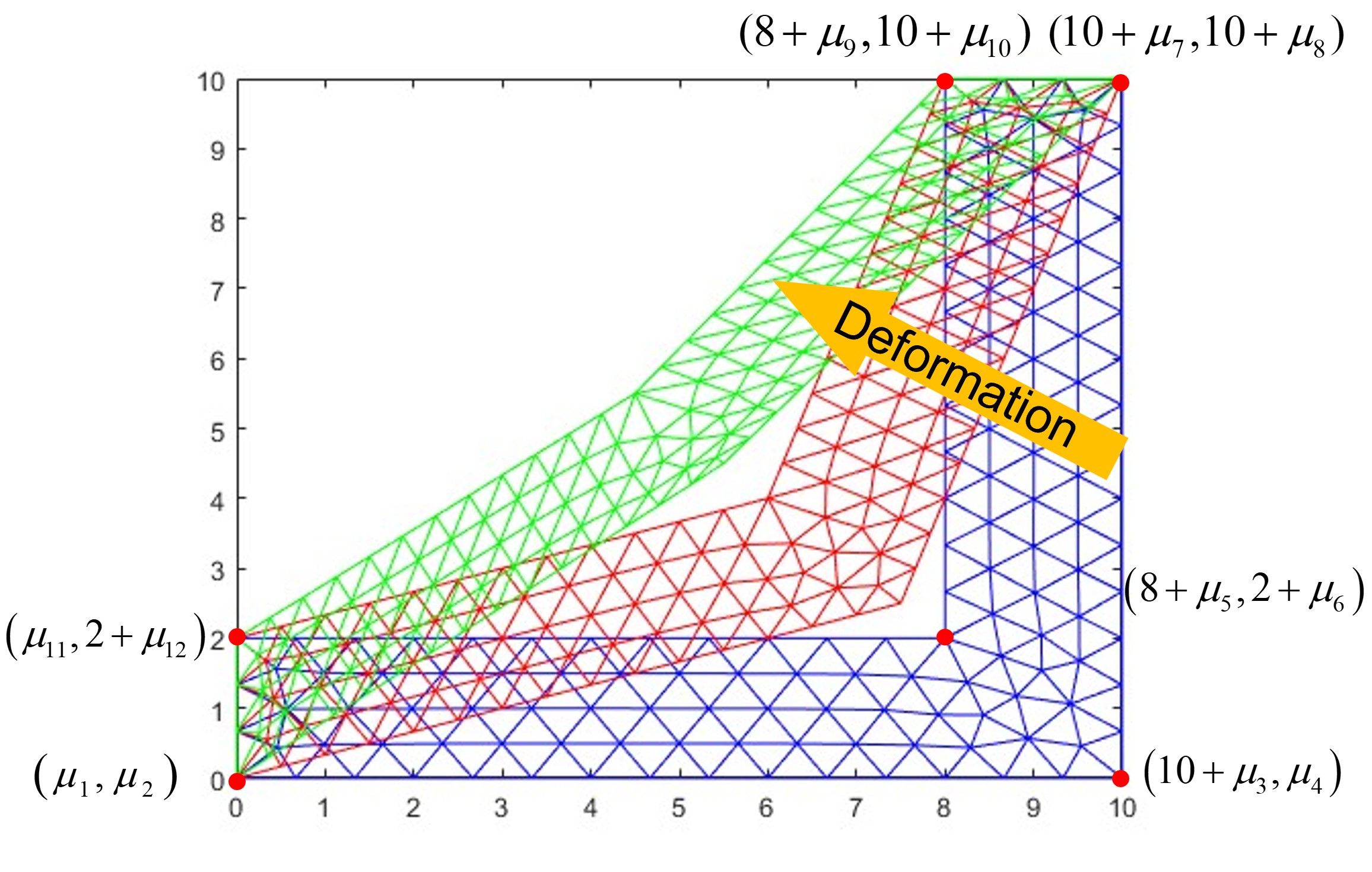}
    \caption{Deformation of an L-shape structure mesh.}
    \label{fig:Deform_L_mesh}
\end{figure}

\clearpage

\printcredits

\section*{Acknowledgements}
Chenzi Wang, Peizhi Yu, Wenshuai Lu, and Zheng You acknowledge support from the National Natural Science Foundation of China (Grant No. U21A6003). Chenzi Wang particularly thanks Dr. Lihong Feng and Prof. Peter Benner for hosting his research visit at the Max Planck Institute for Dynamics of Complex Technical Systems (Magdeburg), which created essential conditions for this international collaboration. The authors gratefully acknowledge Professor Bin Zhou and Dr. Bowen Xing for providing their MEMS accelerometer model and Dr. Wei Bian for sharing his MEMS actuator model, which are instrumental in the numerical validation of our proposed method. The authors specifically acknowledge CasADi~\cite{AndGHetal19} for its powerful symbolic framework, which is crucial to our parametric computations. 

\bibliographystyle{elsarticle-num}
\bibliography{references}

\end{document}